%% file: main.tex
\documentclass[11pt]{article}
\usepackage[T1]{fontenc}
\usepackage[utf8]{inputenc}
\usepackage{lmodern,microtype}
\usepackage[margin=1in]{geometry}
\usepackage{amsmath,amssymb,amsthm,mathtools,booktabs,array}
\usepackage{xurl}
\usepackage{tikz}
\usetikzlibrary{arrows.meta,positioning,calc,decorations.pathreplacing,patterns}
\usepackage[hidelinks,pagebackref]{hyperref}
\newcommand{\doi}[1]{\href{https://doi.org/#1}{doi:\nolinkurl{#1}}}
\renewcommand*{\backref}[1]{}
\renewcommand*{\backrefalt}[4]{\ifcase #1\or\space(Cited on p.~#2.)\else\space(Cited on pp.~#2.)\fi}
\hypersetup{pdftitle={Barrier Log-Concavity of Large Forests},pdfauthor={Vadim E. Levit}}

\newtheorem{theorem}{Theorem}[section]
\newtheorem{lemma}[theorem]{Lemma}

\newtheorem{corollary}[theorem]{Corollary}
\theoremstyle{definition}
\newtheorem{conjecture}[theorem]{Conjecture}
\newtheorem{problem}[theorem]{Problem}
\newtheorem{example}[theorem]{Example}

\theoremstyle{remark}
\newcommand{\E}{\mathbb E}
\newcommand{\Var}{\operatorname{Var}}
\newcommand{\Prob}{\mathbb P}
\newcommand{\I}{\mathcal I}

\newcommand{\ii}{\mathrm i}

\title{Barrier Log-Concavity of Large Forests}
\author{Vadim E. Levit\\[2pt]
\small Department of Mathematics\\
\small Ariel University, Israel\\[2pt]
\small\texttt{levitv@ariel.ac.il}}
\date{October 15, 2026}

\begin{document}
\maketitle
\begin{abstract}
We prove that the independence sequence of every sufficiently large forest is
strictly barrier log-concave: whenever $i_{k+1}>i_k$, one has
$i_k^2>i_{k-1}i_{k+1}$. Consequently, every common downward coefficient tilt is
unimodal, with one order threshold valid simultaneously for all $0<t\le1$.
As an application, we prove that every sufficiently large well-covered tree
has a strictly log-concave independence polynomial. More generally, the
latter conclusion holds for well-covered forests with sufficiently large
matching number. Every sufficiently large forest with a palindromic independence
polynomial is also strictly log-concave, with its modes at the central ranks.
The proof establishes local coefficient-curvature estimates
uniformly over forests of unbounded maximum degree and over activities
tending to zero. Combined with the fixed-density theorem of Fang, Lu,
Nevo, Yao, and Zheng, these estimates complete the strict log-concavity
interval down to rank one through $\lfloor17\alpha/25\rfloor$; the terminal
fraction is inherited. The classical decreasing tail then yields the
barrier property. A cumulative-coefficient corona theorem transfers barrier
log-concavity of the base forest to full strict log-concavity of its leaf
corona.
\end{abstract}
\medskip
\noindent\textbf{2020 Mathematics Subject Classification.} Primary 05C31; Secondary 05C05, 05C69, 60F05.
\smallskip

\noindent\textbf{Keywords.} Independence polynomial; forest; barrier log-concavity; unimodality; well-covered graph; corona; hard-core model.

\section{Introduction and main results}\label{sec:introduction}
For a finite simple graph $G$, let $i_k(G)$ count its independent sets of
size $k$, let $\alpha(G)$ be its independence number, and write
\[
 I(G;x)=\sum_{k=0}^{\alpha(G)}i_k(G)x^k,\qquad
 \Delta_k(G)=i_k(G)^2-i_{k-1}(G)i_{k+1}(G).
\]
The coefficients are positive on $0,\ldots,\alpha(G)$. An internal index
satisfies $1\le k<\alpha(G)$; outside the support the coefficients are zero.
All graph deletions below are induced, and $N_G[A]=A\cup N_G(A)$.
The independence polynomial was introduced by Gutman and Harary
\cite{GutmanHarary1983}.

A positive sequence $a_0,\ldots,a_d$ is \emph{barrier log-concave} (BL) if
\[
 a_{k+1}>a_k\quad\Longrightarrow\quad a_k^2\ge a_{k-1}a_{k+1}
 \qquad(1\le k<d).
\]
It is \emph{strictly} BL if the conclusion is strict at every such ascent.
Ordinary log-concavity (LC) means $a_k^2\ge a_{k-1}a_{k+1}$ at every
internal index. In contrast, BL permits failures on
the nonascending portion of the sequence but rules out a renewed ascent. Indeed, with
$r_k=a_k/a_{k-1}$, BL is equivalent to
$r_{k+1}\le\max\{1,r_k\}$, and hence implies unimodality.

\begin{theorem}[Barrier log-concavity of large forests]\label{cor:barrier}
There exists an absolute integer $N$ such that every forest $F$ of order
$n\ge N$ satisfies
\[
 i_{k+1}(F)>i_k(F)\quad\Longrightarrow\quad
 i_k(F)^2>i_{k-1}(F)i_{k+1}(F)
 \qquad(1\le k<\alpha(F)).
\]
Moreover, for every $0<t\le1$, the coefficient sequence of $I(F;tx)$ is
strictly barrier log-concave and unimodal. The same $N$ works for all such $t$.
\end{theorem}

Eventual forest unimodality was established by Fang, Lu, Nevo, Yao, and
Zheng \cite{FLNYZ}. Theorem~\ref{cor:barrier} strengthens that conclusion:
it controls the determinant at \emph{every ascent} and establishes
unimodality simultaneously for every common downward tilt. It does not
assert full log-concavity or settle BL at all forest orders.

More recently, Zhang and Li \cite[Theorem~1.1 and Propositions~1.2--1.3]{ZhangLi2026}
state unimodality for every finite forest in a computer-assisted proof
combining coverage through order $60$ with an argument from order $59$.
Li, Vallier, and Zhang \cite[Theorem~1.1 and Section~1]{LVZ} present a
second proof, treating orders at most $24$ by exact counting and orders
at least $25$ by a moment argument with finite computational checks.
They explicitly build on Zhang and Li's decomposition and other inputs,
and report a Lean~4 formalization. The latter manuscript is explicitly marked provisional.

Both all-order statements concern the unweighted sequence of $I(F;x)$.
Unimodality alone does not imply the barrier condition or unimodality of
$I(F;tx)$ for every $0<t\le1$; the Gaussian example in
Section~\ref{subsec:beyond-forests} illustrates this distinction for
positive sequences. Neither all-order unimodality theorem is a premise
of the main results below. The separately stated quantitative input
from \cite{LVZ} is used only in the optional supplement.

A graph is \emph{well-covered} if all its maximal independent sets have
cardinality $\alpha(G)$. Write $\mu(F)$ for its maximum matching size and
$z(F)$ for its number of isolated vertices. Our main structural application
reaches full log-concavity.

\begin{theorem}[Log-concavity of large well-covered forests]\label{cor:eventual-wc}
Let $N$ be the threshold in Theorem~\ref{cor:barrier}. Every well-covered
forest $F$ with
\[
 \mu(F)\ge N,\qquad\text{equivalently}\qquad |V(F)|-z(F)\ge2N
\]
has a strictly log-concave independence polynomial at every internal index.
The conclusion is preserved under every common positive coefficient tilt.
In particular, every well-covered forest without isolated vertices of
order at least $2N$, and hence every well-covered tree of that order,
has a strictly log-concave independence polynomial.
\end{theorem}

The hypothesis is on the nonisolated part. A well-covered forest has the
form $(H\circ K_1)\sqcup z(F)K_1$, where $H$ is a forest and
$|V(H)|=\mu(F)$; the classical structural background includes
Levit--Mandrescu \cite{LM1999,LM2004}. The complete proof in
Section~\ref{app:corona} applies BL to the \emph{whole} base forest $H$ and
then a cumulative-coefficient corona transfer. No componentwise lower
order bound and no generic BL product-closure assertion are used.
A second application giving log-concavity at every internal rank concerns
palindromic independence polynomials; Corollary~\ref{cor:palindromic-large} gives strict log-concavity
and exact central modes under the same hypothesis that the order is
sufficiently large.

\subsection{The stronger prefix behind the barrier theorem}
\begin{theorem}[Complete lower prefix]\label{thm:main}
There is an absolute integer $N$ such that every forest of order $n\ge N$
satisfies
\[
 \Delta_j(F)>0\qquad
 \left(1\le j\le\left\lfloor\frac{17\alpha(F)}{25}\right\rfloor\right).
\]
\end{theorem}

The same $N$ may be chosen in Theorems~\ref{cor:barrier}--\ref{thm:main}.
Fang, Lu, Nevo, Yao, and Zheng prove strict log-concavity on
$[\varepsilon n,17\alpha(F)/25]$ for each fixed $\varepsilon>0$ and every
sufficiently large forest \cite[Theorem~1.2 and Section~9]{FLNYZ}.
The fraction $17/25$ and that fixed-density input are inherited.
Our completion covers all lower integer ranks with one order threshold.
The quantifiers do not permit substituting $\varepsilon=1/n$:
$N_{1/n}$ need not be at most $n$.

The new ingredient is the degree-uniform central-curvature estimate in
Theorem~\ref{central:thm:central}. It allows activity to tend to zero and
covers the intermediate ranks missed by the elementary quadratic bound.
The completed prefix overlaps the classical decreasing tail, beginning at
$\lceil(2\alpha(F)-1)/3\rceil$ \cite{LM}. Every possible ascent therefore
lies inside the strictly log-concave prefix, proving Theorem~\ref{cor:barrier}.
Figure~\ref{fig:prefix-cover} illustrates this overlap.

For the prescribed cutoff $q(F)=\lceil(2\alpha(F)-1)/3\rceil$,
prefix log-concavity means
$\Delta_k(F)\ge0$ for every internal $k<q(F)$. Within the class with that
certified tail, full LC implies the prescribed-prefix property, which
implies BL and unimodality. Allowing the cutoff to be chosen afterward
at the first mode removes the additional strength. Ordered (Poisson)
LC, $k i_k^2\ge(k+1)i_{k-1}i_{k+1}$, is a different, stronger property.

\paragraph{Proof architecture.}
Sections~\ref{sec:smallrank}--\ref{sec:centralcurvature} prove the uniform
curvature estimate. Section~\ref{sec:uniformcompletion} completes the
prefix and proves Theorem~\ref{cor:barrier}; Section~\ref{app:corona}
gives the full corona argument and proves Theorem~\ref{cor:eventual-wc}.
The palindromic application follows directly from the completed prefix
by coefficient reflection in Section~\ref{subsec:palindromic}.

\paragraph{Optional quantitative refinements.}
The quantitative supplement \cite{QuantitativeSupplement} contains the
sharper low-activity analysis and explicit refinements. None of its
additional estimates is needed for these main theorems. The optional
explicit BL cutoff in that supplement retains a separately stated input
from Li, Vallier, and Zhang \cite{LVZ}; the main proof here does not use it.

\subsection{Relation to earlier methods}
The root recursion, bipartition damping, and Fourier second-difference
framework are adapted from \cite[Sections~4--6]{FLNYZ}; the occupation-odds
recursion also underlies Weitz's tree method \cite{Weitz}.
Related response recursions and innovation-variance identities appear in
Zhang--Li \cite[Section~7.1, (83)--(86)]{ZhangLi2026} and
Li--Vallier--Zhang \cite[Section~4.3, Lemmas~4.24--4.25]{LVZ}.
Those identities are not claimed as new here. The additional weighted
operator-norm and predictable-clock estimates control third moments,
clock fluctuations, and hence curvature uniformly in activity and
maximum degree. A point-probability error
$o(\sigma^{-1})$ alone does not decide a second difference of order
$\sigma^{-3}$.

Vanishing activity is already allowed in the bounded-degree local limit
theorem of Jain, Perkins, Sah, and Sawhney \cite[Theorem~1.5]{JPSS}:
for fixed $\Delta\ge3$ and $0<\delta<1$, it allows $\lambda_n\to0$ with
$n\lambda_n\to\infty$ and
$\lambda_n\le(1-\delta)\lambda_c(\Delta)$, where
$\lambda_c(\Delta)=(\Delta-1)^{\Delta-1}/(\Delta-2)^\Delta$.
Jain, Michelen, Pham, and Vuong \cite[Proposition~10]{JMPV} give
high-order Edgeworth expansions at bounded degree with a fixed positive
lower activity bound. Our contribution is simultaneous degree/activity
uniformity at the coefficient-curvature scale, not vanishing activity
or high-order approximation alone. Classical antecedents include Bender
\cite{Bender} and Dobrushin--Tirozzi \cite{DT}; zero-based criteria
\cite{LPRS,MS} impose different hypotheses. Pre-Lorentzian certificates
\cite{BendjeddouHardiman,LiuTang} concern specified graph constructions.

The supplement records additional variance bounds in relation to
Davies--Sandhu--Tan \cite{DST} and Zhang--Xu \cite{ZhangXu}, as well as a
normal-approximation consequence using Rai\v{c}'s explicit Berry--Esseen
bound \cite{Raic}. Those auxiliary refinements are not hidden inputs to
the proof below. The sharp-prefix companion \cite{LevitDraft} proves the
first twelve LC inequalities at every forest order; its certificates are
not premises here.

\subsection{Earlier partial unimodality and special graph classes}\label{sec:historical-partial}
The distinction between a theorem about a whole sequence and one about
prescribed parts of it has a substantial history. Alavi, Malde, Schwenk,
and Erd\H{o}s \cite{AMSE1987} posed the tree/forest unimodality question
in the setting of essentially unrestricted coefficient orderings for
general graphs. The surveys of Levit and Mandrescu
\cite{LMSurvey2005,LMPartial2006} organize the earlier results, including
\emph{partial unimodality}: a nondecreasing initial segment and a
nonincreasing terminal segment, with a possibly unresolved middle.
That terminology does not assert log-concavity of the initial segment
or the BL implication used here. We record the principal deterministic
bounds and the relevant positive families separately.

\paragraph{Initial and terminal inequalities.}
Write $a=\alpha(G)$ and $n=|V(G)|$ in this historical discussion.
Michael and Traves \cite{MichaelTraves2003} proved
$i_0\le\cdots\le i_{\lceil a/2\rceil}$ for well-covered graphs.
The failure of full unimodality for general well-covered graphs,
and the eventual resolution of the permitted coefficient-ordering
problem by Cutler and Pebody \cite{CutlerPebody2017}, make this a
proper partial result, not an all-rank theorem. Levit and Mandrescu
\cite{LMGeneric2006} also construct connected well-covered counterexamples
for every independence number at least eight. Levit and Mandrescu's
very well-covered, quasi-regularizable, and perfect-graph results
\cite{LMParis2007}, together with their K\"onig--Egerv\'ary formulation
\cite{LM}, give the terminal inequality
\[
 i_{\lceil(2a-1)/3\rceil}\ge\cdots\ge i_a
\]
for forests; for very well-covered graphs it accompanies the initial
segment through $\lceil a/2\rceil$. That work also proves unimodality
for very well-covered graphs with $a\le9$ and LC for those with
$a\le5$. The perfect-graph statement has the
clique-number cutoff $\lceil(\omega a-1)/(\omega+1)\rceil$.
The 2007 chapter develops the earlier 2004 preprint; these are not
counted as two independent discoveries.

Levit and Mandrescu \cite[Theorem~2.1]{LMRoller2017} also prove that if
$\kappa|A|\le|N(A)|$ for every independent set $A$, where $\kappa>0$,
then
\[
 (k+1)i_{k+1}\le\bigl(n-(\kappa+1)k\bigr)i_k.
\]
This gives a nonincreasing tail beginning at
$\lceil(n-1)/(\kappa+2)\rceil$ when that index lies in the support.
We use $\kappa$, not the source's $\lambda$, to avoid confusing this
expansion parameter with physical activity. For connected
$1$-well-covered graphs on more than two vertices, their results give
an initial segment through $\lceil2a/3\rceil$ and a terminal segment
from $\lceil(n-1)/3\rceil$. Here $1$-well-covered means that the graph
and every single-vertex deletion are well-covered. Galvin
\cite{Galvin2012} studies other partial-unimodality regions in regular
bipartite graphs; those hypotheses should not be replaced by
bipartiteness alone.

Basit and Galvin \cite[Theorems~3, 5, and 6]{BG} prove a
nonincreasing tail from $\lceil a(n-1)/(n+a)\rceil$ for every graph.
If every maximal independent set has size at least $\beta$, they prove
an initial segment through $\lceil\beta/2\rceil$; for an $n$-vertex
tree this yields $\lceil(n-a+1)/4\rceil$. These are prior monotonicity
bounds, distinct from the newer forest initial segment through
$\lceil n/4\rceil$ and the fixed-density LC interval in \cite{FLNYZ}.
None of these monotone-segment conclusions by itself asserts a
nonnegative Tur\'an determinant at all of its indices.

\paragraph{Graph operations and well-covered constructions.}
The earlier graph-operation theorem of Levit and Mandrescu
\cite{LMGraphOps2008} treats $G=H\circ 2K_1$, obtained by attaching two
private leaves at each vertex of $H$. With $a=\alpha(G)$, not
$\alpha(H)$, its stated partial ranges are
\[
 i_0\le\cdots\le i_{\lfloor(2a+2)/5\rfloor},\qquad
 i_{\lceil(3a-2)/5\rceil}\ge\cdots\ge i_a.
\]
Their symmetry work \cite{LMPal2007,LMSym2011} and Mandrescu's later
palindromicity argument \cite{Mandrescu2012} belong to this line of
research; the latter proves full unimodality for $H\circ2K_1$.
Symmetric tree constructions and their relation to $\gamma$-positivity
are studied in \cite{HKV}; further comparisons appear in the quantitative
supplement. These operations differ from the single-leaf corona
$H\circ K_1$ used in the present comparison.

For the single-leaf corona, Levit and Mandrescu
\cite{LM2003Family} proved unimodality when the \emph{base} has
independence number at most four; Chen and Wang \cite{ChenWang2010}
extended that condition to eight. The early well-covered-tree papers
\cite{LM2002,LM2003LNCS,LM2004} supply positive families including
well-covered spiders and centipedes. Further coefficient identities and root correspondences for the
single-leaf corona appear in \cite{LMRoots2008,LMCorona2016}. Brown and Cameron
\cite{BC2018} prove strict LC after sufficiently many iterated leafy
extensions of a fixed graph, while Galvin and Hilyard
\cite[Proposition~1.10]{GalvinHilyard2018} obtain LC by attaching
sufficiently many leaves to every vertex of a fixed graph. The number
of iterations, the number of leaves, the independence number of the
base, and its total order are different parameters; the corresponding
hypotheses are not interchangeable. Hoang, Levit, Mandrescu, and Pham
\cite{HLMPUnimodal2026} establish coefficient inequalities for $W_p$
graphs and unimodality of clique coronas of stars. These results do not
assert full LC for every single-leaf corona.

\paragraph{Tree families and stronger algebraic certificates.}
Hamidoune's claw-free LC theorem \cite{Hamidoune1990} and the
real-rootedness theorem of Chudnovsky and Seymour
\cite{ChudnovskySeymour2007} are important tools behind several
constructions. A claw-free tree is a path, so neither theorem covers
arbitrary trees. Z.-F.~Zhu \cite{ZFZhu2007} treats centipedes and
related families. Wang and B.-X.~Zhu \cite{WangZhu2011} use
factorizations for graph concatenations, including specified regular
caterpillars. Galvin and Hilyard \cite{GalvinHilyard2018} establish
LC for alternating pendant-decoration families. Bahls, Ethridge,
and Szabo \cite{BahlsEthridgeSzabo2018} prove unimodality for
nonregular caterpillars satisfying explicit ordering and growth
conditions on their leaf counts, not for all caterpillars.

The path-like constructions of Bahls and Salazar
\cite{BahlsSalazar2010}, Bahls \cite{Bahls2012}, and Bahls, Bailey,
and Olsen \cite{BahlsBaileyOlsen2014} give further symmetric-unimodal,
LC, or real-rooted families, according to the particular construction.
Bencs \cite{Bencs2018} uses stable-path-tree methods to prove
real-rootedness for families including centipedes and Fibonacci trees.
Partial-order and direct LC approaches for specified tree families
appear in \cite{BRGG2019,BRGG2020,BZhuChen2019}.
Mandrescu and Spivak \cite{MandrescuSpivak2016} prove LC for the
odd-order trees maximizing the number of maximal independent sets
and for certain even-order cases; ``maximal trees'' there is an
extremal class, not all trees.

Graph-building and rooted-product methods supply additional families
\cite{Mandrescu2009,BZhuClique2016,BZhuWang2020,ZhuZhu2023,ZhuZhu2024,LiuTangZhao2026}.
Their hypotheses concern the chosen building blocks and rooted states;
none licenses arbitrary tree gluing. Bendjeddou and Hardiman
\cite{BendjeddouHardiman} give Lorentzian certificates for specified
graph constructions. More recently, Li, Li, Yang,
and Zhang \cite{LLYZ2026} prove LC for \emph{all spiders}, that is,
trees with at most one vertex of degree greater than two, by a
chromatic-symmetric-function method. Grace Li
\cite[Theorems~1.4--1.5]{LiKL2026} proves unimodality for both
$T_{3,m,n}$ and $T^*_{3,m,n}$ with $m,n\ge1$, the two source-defined
families containing the original tree LC counterexamples.
That positive result is important context: an LC failure within such
a family does not constitute a unimodality failure. It is not cited
here as a general BL or all-forest theorem.

\paragraph{Random trees and finite evidence.}
Heilman \cite{Heilman2025} obtains high-probability partial unimodality
for random trees and separate sparse-random-graph models. For a
uniform random labeled tree, the initial range is approximately
$46.8\%$ of the nonzero sequence. Basit and Galvin
\cite[Theorems~4 and 7]{BG} obtain, asymptotically almost surely,
ascent through about $0.280n$ and descent from about $0.347n$.
These are fractions of \emph{vertex order}, not $28\%$ and $34.7\%$
of the nonzero coefficient support. Both projects circulated in 2020;
Heilman's 2025 journal date is not the date of a later improvement
over the 2021 Basit--Galvin paper. High-probability statements do not
replace a theorem for every tree.

Yosef, Mizrachi, and Kadrawi \cite{YMK2022} verified LC through order
20. The stronger computation through order 25 reported in
\cite{BG,Ball}, and the order-26 calculation and exceptions in
\cite{KLYM,KL26}, have different scopes. Those finite records must
not be merged with random-tree results, with theorems for restricted
infinite classes, or with the recent eventual and all-forest
unimodality arguments \cite{FLNYZ,ZhangLi2026,LVZ}.

\paragraph{Complementary dependence polynomials.}
For a graph $G$ of order $n$ containing at least one edge, the
\emph{dependence polynomial}
\[
 D_{\mathrm{dep}}(G;x)=(1+x)^n-I(G;x)
\]
counts vertex sets whose induced subgraph contains an edge. Horrocks
\cite{Horrocks2002} proved that these counts are log-concave for every graph.
Xie and Xu \cite[Theorems~1.2--1.3]{XieXu2026} prove ultra log-concavity
when $G$ is $(K_2\sqcup2K_1)$-free or has an independent set of size
$n-2$, and characterize the real-rooted dependence polynomials. Here
ultra log-concavity uses normalization by $\binom nk$ at the ambient
order $n$; the zero coefficients at ranks zero and one are retained.
Complementarity does not transfer these conclusions to $I(G;x)$.
Indeed, let $G=K_6\vee6K_1$, where $\vee$ adds every edge between the
two disjoint parts. Every edge has a universal endpoint, so $G$ is
$(K_2\sqcup2K_1)$-free and its dependence polynomial is ultra
log-concave by the cited theorem. Nevertheless,
$I(G;x)=(1+x)^6+6x$ has $(i_1,i_2,i_3)=(12,15,20)$, giving
$\Delta_2=15^2-12\cdot20=-15$ at an ascent. This general-graph
example is not a forest counterexample. Dependence counts are also
distinct from domination counts.

Theorem~\ref{cor:barrier} controls all ascents of every sufficiently
large forest; Theorem~\ref{thm:main} supplies the stronger prefix
that proves it. The
fixed-density endpoint and the classical tail remain credited above.
The forest-corona consequence in Section~\ref{app:corona} uses the
complete transfer proved there, rather than any claim that all
well-covered graphs or all single-leaf coronas are already LC.
The historical results in this subsection are context, not additional
premises of the central-curvature proof.

\section{Forest extension bounds and hard-core estimates}\label{sec:smallrank}
At physical activity $\lambda>0$, choose an independent set $I$ with
probability proportional to $\lambda^{|I|}$ and put
\[
 Z_F(\lambda)=I(F;\lambda),\quad X=|I|,\quad
 p_j=\frac{i_j(F)\lambda^j}{Z_F(\lambda)},\quad
 \eta_F(\lambda)=\E_\lambda X,\quad \sigma^2=\Var_\lambda X,\quad \nu=n\lambda.
\]
Write $\eta=\eta_F(\lambda)$. The activity mass $\nu$ is not the matching
number $\mu(F)$. Division by $\sigma$ always concerns a nonempty forest.
Conditioning the common-activity law on cardinality gives the uniform
independent-set law at that rank. A uniformly sampled extension incidence
instead size-biases its predecessor by the number of legal extensions;
that law is not substituted into the rank averages below.

\subsection{An elementary quadratic bound}
The extension statistic is standard \cite{BG,FLNYZ}; we include the
needed argument so that no finite certificate is an input here.
Let $\I_r(F)$ be the family of independent $r$-sets. Put
\[
 m_r=(r+1)\frac{i_{r+1}}{i_r}\quad(0\le r<\alpha(F)),
 \qquad m_{\alpha(F)}=0.
\]
\begin{lemma}\label{lem:extension}
For every forest, $n-3r\le m_r\le n-r$ and
$m_{r+1}\ge m_r-3$ for $0\le r<\alpha(F)$. At every internal index,
\[
 n\ge2j^2-4\quad\Longrightarrow\quad\Delta_j>0.
\]
\end{lemma}
\begin{proof}
For uniform $A\in\I_r(F)$, write $D_r=|F-N[A]|$.
One- and two-extension counting give
\[
 m_r=\E D_r,\qquad
 m_rm_{r+1}=2\E i_2(F-N[A]).
\]
A forest of order $d$ satisfies $2i_2\ge g(d)$, where $g$ is zero on
$[0,2]$ and $g(d)=(d-1)(d-2)$ on $[2,\infty)$. This function is convex,
so Jensen gives
\[
 m_{r+1}\ge m_r-3+\frac2{m_r}\qquad(m_r\ge2).
\]
If $0<m_r<2$, the weaker inequality $m_{r+1}\ge m_r-3$ is automatic.
Since $m_0=n$, induction proves $m_r\ge n-3r$, including the terminal
value. The bound $m_r\le n-r$ follows by counting available vertices.

Rank one is immediate from $i_1=n$ and $i_2\le\binom n2$.
For $j\ge2$ and $n\ge2j^2-4$, the preceding means used in the induction
are at least two: for $0\le r\le j-2$,
$n-3r\ge2j^2-3j+2\ge4$.
Consequently,
\[
 \frac{i_j}{i_{j-1}}\ge
 \frac{n-3j+3+2(j-1)/n}{j},\qquad
 \frac{i_{j+1}}{i_j}\le\frac{n-j}{j+1}.
\]
The difference of these bounds has numerator
$n^2+(3-2j^2)n+2(j^2-1)$ over $j(j+1)n$.
Putting $n=2j^2-4+u$ makes the numerator
$2+(2j^2-5)u+u^2>0$.
\end{proof}

\begin{lemma}[Mean bounds]\label{lem:mean}
For every nonempty forest and every $\lambda>0$,
\[
 \frac{n\lambda}{1+3\lambda}\le\eta_F(\lambda)
 \le\frac{n\lambda}{1+\lambda}\le n\lambda.
\]
\end{lemma}
\begin{proof}
Conditioning a hard-core independent set on its size $r$ gives the
uniform law on $\I_r(F)$. Weighted extension counting gives
$\lambda\E m_X=\E X=\eta$. Averaging both bounds in
Lemma~\ref{lem:extension} and rearranging proves the claim.
\end{proof}

\subsection{Bipartition damping}
\begin{lemma}[Bipartition variance and damping]\label{lem:bipartition}
For a nonempty forest and $0<\lambda\le1$,
\[
 \sigma^2\ge\frac{n\lambda}{2(1+\lambda)^4},\qquad
 |\E_\lambda e^{\ii sX}|\le
 \exp\left\{-\frac{n\lambda}{2(1+\lambda)^3}\sin^2(s/2)\right\}
 \quad(s\in\mathbb R).
\]
\end{lemma}
\begin{proof}
For a bipartition $L\sqcup R$, condition on $I\cap R$. Each available
vertex of $L$ is occupied independently with probability
$\lambda/(1+\lambda)$. Total variance, averaged over the two possible
conditioning classes, gives
\[
 \sigma^2\ge\frac{\lambda}{2(1+\lambda)^2}
       \sum_{v\in V(F)}\Prob(N(v)\cap I=\varnothing).
\]
A vertex has occupation probability at most $\lambda/(1+\lambda)$,
also after conditioning specified other vertices to be absent. Hence
sequential absence conditioning gives
$\Prob(N(v)\cap I=\varnothing)\ge(1+\lambda)^{-d(v)}$.
Jensen's inequality and the forest average-degree bound prove the
variance claim. This is the activity-dependent version of
\cite[(5.3)]{FLNYZ}.

For damping, choose $|L|\ge n/2$ and let $W_L$ be the number of available
vertices of $L$ after conditioning on $R$. Conditional independence gives
\[
 \left|\E(e^{\ii sX}\mid I\cap R)\right|
 \le\left(1-\frac{4\lambda}{(1+\lambda)^2}\sin^2(s/2)\right)^{W_L/2}
 \le h^{W_L},\qquad
 h=1-\frac{\lambda}{1+\lambda}\sin^2(s/2).
\]
The last inequality uses $\lambda\le1$. Write $Z(u,\lambda)$ for the
partition function with activity $u$ on $L$ and $\lambda$ on $R$, and
put $\tau=(1+\lambda)h-1=\lambda\cos^2(s/2)$. Then
\[
 \E h^{W_L}=\frac{Z(\tau,\lambda)}{Z(\lambda,\lambda)},\qquad
 \frac{\mathrm d}{\mathrm du}\log Z(u,\lambda)
 \ge\frac1{1+u}\sum_{v\in L}(1+\lambda)^{-d(v)}.
\]
Integrating from $\tau$ to $\lambda$, using continuity at zero, yields
$\E h^{W_L}\le h^{A_L}$, where
$A_L=\sum_{v\in L}(1+\lambda)^{-d(v)}$.
Since $\sum_{v\in L}d(v)=|E(F)|\le n$ and $|L|\ge n/2$,
Jensen gives $A_L\ge n/[2(1+\lambda)^2]$. Now use
$|\E e^{\ii sX}|\le\E h^{W_L}$ and $\log h\le-(1-h)$.
This preserves the activity dependence in the argument of
\cite[Lemma~6.1]{FLNYZ}.
\end{proof}

\section{Uniform coefficient curvature}
\label{sec:centralcurvature}
\begin{theorem}[Uniform central curvature through activity one]
\label{central:thm:central}
For every forest, if $0<\lambda\le1$, $\nu=n\lambda\ge2^{104}$, and
$\eta=k$ is an integer, then
\[
 \left|\sigma^3(2p_k-p_{k-1}-p_{k+1})-
                    \frac1{\sqrt{2\pi}}\right|<\frac{39}{100}.
\]
In particular, the central second difference is strictly positive, and
$\Delta_k(F)>0$.
\end{theorem}

The next operator bounds allow unequal vertex activities. We then
specialize to a common activity to control the characteristic function
and its Fourier second difference. All estimates in this section are
proved here; the arithmetic certificates check only their explicit
constant comparisons.

For a rooted forest with vertex activities $0<\lambda_v\le1$, let $q_v$
be the occupation probability at $v$ in its descendant subtree, with its
parent, when present, conditioned to be vacant. Put $b_v=q_v(1-q_v)$ and let
\[
 (\mathsf{M}z)_v=\sum_{u\text{ child of }v}q_u z_u,
 \qquad \|z\|_{p,b}^p=\sum_v b_v|z_v|^p.
\]
We use the same subscript for the corresponding induced operator norm.
The hard-core recursion gives
\[
 \frac{q_v}{1-q_v}=\lambda_v\prod_{u\text{ child of }v}(1-q_u),
 \qquad 0<q_v\le\frac12.
\]

\begin{lemma}[Response operator bounds]\label{lem:cubic-response-norms}
For every finite rooted forest with vertex activities
$0<\lambda_v\le1$, the following bounds hold:
\[
 \|\mathsf{M}\|_{3/2,b}<\frac12,
 \qquad \|\mathsf{M}\|_{2,b}\le\sqrt{\frac8{25}}<\frac35,
 \qquad \|\mathsf{M}\|_{3,b}\le2^{-1/3}<\frac45.
\]
\end{lemma}
\begin{proof}
Weighted H\"older at each parent gives
\begin{equation}\label{cubic:eq:1}
\|\mathsf{M}\|_{p,b}^p\le
 \sup_v b_v\left(
 \sum_{u\text{ child of }v}\frac{q_u}{(1-q_u)^{1/(p-1)}}
 \right)^{p-1}.
\end{equation}
Write $y_i=q_i/(1-q_i)\in[0,1]$, $R=\prod_i(1+y_i)$,
and $S=\sum_i y_i$ at a parent.  Its weight is at most
$R/(1+R)^2$, since $\lambda_v\le1$.

For $p=3/2$, the square of the right-hand local coefficient in \eqref{cubic:eq:1} is
at most
\[
 \frac{R^2}{(1+R)^4}\sum_i y_i(1+y_i).
\]
For a fixed sum $S=k+u$, $k\ge0$ an integer and $0\le u\le1$,
moving two non-endpoint entries toward $0,1$ decreases their factor
product.  Hence $R\ge2^k(1+u)\ge2(k+u)=2S$; the second inequality is
immediate for $k=0$, and follows from $2^k\ge2k$ for $k\ge1$.
Consequently, $\sum_i y_i(1+y_i)\le2S\le R$, and
\[
 \|\mathsf{M}\|_{3/2,b}^3\le
 \sup_{R\ge1}\frac{R^3}{(1+R)^4}=\frac{27}{256}<\frac18.
\]

For $p=2$, \eqref{cubic:eq:1} is $b_vS$.  The sharp local bound follows by endpoint packing.
For a fixed sum $S=k+u$, endpoint packing gives
$R\ge2^k(1+u)$.  Since $R/(1+R)^2$ decreases on $R\ge1$,
it remains to maximize
\[
 f_k(R)=\frac{R(k-1+R/2^k)}{(1+R)^2},
 \qquad 2^k\le R\le2^{k+1}.
\]
Its derivative has the same sign as
$k-1+(2^{1-k}-(k-1))R$.  This is positive for $k=0,1$ and
negative for $k\ge2$ throughout their respective intervals.
The maximum is therefore at $S=2$, $R=4$, and equals $8/25$.

For $p=3$, set $H=\log R=\sum_i z_i$, where
$z_i=\log(1+y_i)\in[0,\log2]$.  The convexity of
$2\sinh(z/2)$ and its endpoint chord give
\[
 \sum_i\frac{q_i}{\sqrt{1-q_i}}
 =\sum_i2\sinh(z_i/2)
 \le\frac{H}{\sqrt2\log2}.
\]
Thus the local coefficient is at most
\begin{equation}\label{cubic:eq:2}
\frac{H^2}{4(\log2)^2(1+\cosh H)}<\frac12.
\end{equation}
For an explicit verification of the strict inequality, use
$\log2>11/16$ and $\cosh H\ge1+H^2/2+H^4/24$.
After multiplying by $3072$, the needed positive expression is
\[
 121H^4-1620H^2+5808>0,
\]
whose quadratic discriminant in $H^2$ is $-186672$.
Finally, $(4/5)^3>1/2$.  This proves all three assertions.
\end{proof}

For a linear statistic $Y_a=\sum_v a_vJ_v$, where $J_v$ is the occupied
indicator in the whole forest and $a_v$ are fixed real weights, let
$d_v^a$ be the occupied-root conditional mean minus the absent-root
conditional mean of $\sum_{u\in T_v}a_uJ_u$ in the descendant tree
$T_v$. These responses obey
\begin{equation}\label{cubic:eq:3}
(\mathrm{Id}+\mathsf{M})d^a=a.
\end{equation}
Let $\mathcal R$ be the set of roots. Reveal vertices with parents
before children, and define
\[
 \zeta_r=J_r-q_r\quad(r\in\mathcal R),\qquad
 \zeta_v=J_v-q_v(1-J_{\operatorname{par}(v)})\quad(v\in V(F)-\mathcal R).
\]
Conditional on the previously revealed vertices, an unrevealed descendant
subtree interacts with that history only through its parent. Thus the
conditional mean of $J_v$ is $q_v(1-J_{\operatorname{par}(v)})$ at a
nonroot, and $q_v$ at a root. Consequently, the $\zeta_v$ are martingale
differences. The response recursion telescopes to
\begin{equation}\label{cubic:eq:4}
Y_a-\E Y_a=\sum_v d^a_v\zeta_v.
\end{equation}
Put $c_r=1$ for each root and $c_v=\Prob(J_{\operatorname{par}(v)}=0)$ otherwise.
Then $\E \zeta_v^2=c_vb_v$ and
\[
 \Var Y_a=\sum_v c_vb_v(d^a_v)^2.
\]

\begin{lemma}[Cubic increment and variance-clock bounds]
\label{lem:cubic-clock}
Let $X=\sum_vJ_v$, $d=d^{\boldsymbol1}$, and $B_0=\sum_vb_v$.
Write $\Delta_v=d_v\zeta_v$ for its martingale increments, and
$A$ for their total predictable quadratic variation.  Then
\begin{align*}
 \sum_vb_v|d_v|^3&\le125B_0,\\
 \sum_v\E |\Delta_v|^3&\le125B_0,\\
 \E |A-\Var X|^{3/2}&\le250B_0.
\end{align*}
Also $\Var X\le(25/4)B_0$.
\end{lemma}
\begin{proof}
The $\ell^3(b)$ resolvent estimate in \eqref{cubic:eq:3} gives
$\|d\|_{3,b}\le5\|\boldsymbol1\|_{3,b}$.
The conditional third absolute moment of $\zeta_v$ is at most $b_v$,
since, when its parent is vacant,
\[
 \E |J_v-q_v|^3
 =b_v\bigl(q_v^2+(1-q_v)^2\bigr)\le b_v.
\]
This proves the first two bounds.

For the third, retain the exact clock identity
\begin{equation}\label{cubic:eq:5}
A=\sum_v b_vd_v^2
       -\sum_v h_vJ_v,
 \qquad h_v=\sum_{u\text{ child of }v}b_ud_u^2.
\end{equation}
Indeed, an increment at a nonroot $u$ has conditional variance
$b_ud_u^2(1-J_{\operatorname{par}(u)})$, whereas a root contributes
$b_ud_u^2$.
Now $h=\mathsf{M}((1-q)d^2)$.  If $g=(\mathrm{Id}+\mathsf{M})^{-1}h$ denotes the responses
of $\sum h_vJ_v$, the $\ell^{3/2}(b)$ operator bound gives
\begin{equation}\label{cubic:eq:6}
\|g\|_{3/2,b}\le2\|h\|_{3/2,b}
 \le\|(1-q)d^2\|_{3/2,b}
 \le\|d\|_{3,b}^{2}.
\end{equation}
We use the elementary martingale inequality at $p=3/2$,
\begin{equation}\label{cubic:eq:7}
\E \left|\sum_i Z_i\right|^{3/2}
 \le\sqrt2\sum_i\E |Z_i|^{3/2},
\end{equation}
for martingale differences $Z_i$.  It follows by successive conditional
expectation from
\[
 |x+y|^{3/2}\le |x|^{3/2}
  +\tfrac32\operatorname{sgn}(x)|x|^{1/2}y
  +\sqrt2|y|^{3/2}.
\]
The last scalar estimate follows by integrating the $1/2$-H\"older
bound for $\operatorname{sgn}(x)|x|^{1/2}$, whose constant is $\sqrt2$.
Moreover
\[
 \E |\zeta_v|^{3/2}
 =c_vb_v\bigl(\sqrt{q_v}+\sqrt{1-q_v}\bigr)
 \le\sqrt2 b_v.
\]
Apply \eqref{cubic:eq:7} to the martingale representation of the statistic in \eqref{cubic:eq:5},
then use \eqref{cubic:eq:6}:
\[
 \E |A-\E A|^{3/2}
 \le2\|g\|_{3/2,b}^{3/2}
 \le2\|d\|_{3,b}^3\le250B_0.
\]
Finally, $\E A=\Var X$.
The $\ell^2(b)$ resolvent estimate with norm less than $3/5$ gives
$\Var X\le\|d\|_{2,b}^2\le(25/4)B_0$.
\end{proof}

The proof below uses a classical device for martingale approximation:
stop at a target predictable quadratic variation and complete the remaining
variance. Related clock-completion constructions appear, for example, in
Mourrat \cite[Section~2]{Mourrat}. We include the characteristic-function
estimate and its constants in full; the device itself is not claimed as new.

\begin{lemma}[Direct martingale characteristic-function comparison]
\label{lem:martingale-cf-clock}
Let $S=\sum_{i=1}^m D_i$ be a real martingale difference sum with finite
third absolute moments.  Set
\[
 v_i=\E (D_i^2\mid\mathcal F_{i-1}),\quad
 A_i=\sum_{j\le i}v_j,\quad
 A=A_m,\quad \sigma^2=\E A>0.
\]
For every real $t$,
\begin{equation}\label{cubic:eq:8}
\left|\E e^{itS}-e^{-\sigma^2t^2/2}\right|
 \le\frac{t^2}{2}\E |A-\sigma^2|
 +\frac{1+\E |Z|^3}{6}|t|^3
       \sum_i\E |D_i|^3,
\end{equation}
where $Z$ is a standard normal variable.
\end{lemma}
\begin{proof}
First suppose $A=\sigma^2$ almost surely. Put $S_0=A_0=0$. For the
partial sums $S_i$, telescope
\[
 F_i=\exp\{itS_i-t^2(\sigma^2-A_i)/2\}.
\]
More explicitly, set
\[
 H_i=\exp\{itS_{i-1}-t^2(\sigma^2-A_i)/2\},\qquad
 F_i-F_{i-1}=H_i\bigl(e^{itD_i}-e^{-t^2v_i/2}\bigr).
\]
Since $A_i$ is $\mathcal F_{i-1}$-measurable and $A_i\le\sigma^2$,
$H_i$ is predictable and has modulus at most one. Taylor comparison
with a centered normal variable of variance $v_i$, together with the
conditional Lyapunov inequality, gives
\[
 \left|\E (e^{itD_i}\mid\mathcal F_{i-1})
                  -e^{-t^2v_i/2}\right|
 \le\frac{1+\E |Z|^3}{6}|t|^3
                 \E (|D_i|^3\mid\mathcal F_{i-1}).
\]
Summing proves \eqref{cubic:eq:8} with zero first term.  For a conditionally Gaussian increment, the comparison error is zero.

For general $A$, let $\tau$ be the first index with $A_\tau>\sigma^2$,
when such an index exists. Retain all increments before $\tau$, replace $D_\tau$
by $\alpha D_\tau$, where
\[
 \alpha=\sqrt{\frac{\sigma^2-A_{\tau-1}}{v_\tau}}\in[0,1),
\]
and discard the subsequent increments.  If no crossing occurs, retain all increments. The event that the crossing
occurs at index $i$, and the scaling used there, are $\mathcal F_{i-1}$-measurable.
Regard each discarded increment as zero while retaining the original
filtration through $\mathcal F_m$. The retained clock is $\min(A,\sigma^2)$.
At a final step append $\sqrt{(\sigma^2-A)_+}\,Z$, where $Z$ is a
standard normal variable independent of $\mathcal F_m$. This increment is
conditionally centered Gaussian; its random variance is
$(\sigma^2-A)_+$. The completed martingale has deterministic total
clock $\sigma^2$.

On the event $\{\tau\le m\}$, the discarded suffix satisfies
\[
 \E(S-S_\tau\mid\mathcal F_\tau)=0,\qquad
 \E\bigl((S-S_\tau)^2\mid\mathcal F_\tau\bigr)
       =\E(A-A_\tau\mid\mathcal F_\tau).
\]
The conditional second-order Taylor estimate therefore shows that
removing this suffix changes the characteristic function by at most
$t^2\E[\boldsymbol1_{\{\tau\le m\}}(A-A_\tau)]/2$.
Scaling the crossing increment costs at most
$t^2\E [(1-\alpha^2)v_\tau
\boldsymbol1_{\{\tau\le m\}}]/2$.
For this step integrate
\[
 \left|\E \bigl[itD_\tau
                 (e^{istD_\tau}-1)\mid\mathcal F_{\tau-1}\bigr]\right|
 \le t^2s v_\tau\qquad(\alpha\le s\le1).
\]
The two costs sum to $t^2\E (A-\sigma^2)_+/2$.
Appending the Gaussian costs at most
$t^2\E (\sigma^2-A)_+/2$.
Apply the deterministic-clock comparison to the completed martingale.
Scaling reduces third absolute moments, and the appended Gaussian
contributes no comparison error.  This proves \eqref{cubic:eq:8}.
\end{proof}

\begin{proof}[Proof of Theorem~\ref{central:thm:central}]
At common activity $\lambda$, monotonicity of $q(1-q)$ on $[0,1/2]$
gives $B_0\le\nu/(1+\lambda)^2$.
Lemma~\ref{lem:bipartition} gives
\begin{equation}\label{cubic:eq:9}
\sigma^2\ge\frac{\nu}{2(1+\lambda)^4},\qquad
 |\E e^{i\theta X}|
 \le\exp\left\{-\frac{\nu\sin^2(\theta/2)}{2(1+\lambda)^3}\right\}.
\end{equation}
Lemma~\ref{lem:cubic-clock} supplies the variance upper bound
$\sigma^2\le25\nu/[4(1+\lambda)^2]$.
For $\chi(s)=\E  e^{is(X-k)/\sigma}$ and $|s|\le\pi\sigma$,
using $\sin(|s|/(2\sigma))\ge |s|/(\pi\sigma)$ yields
\begin{equation}\label{cubic:eq:10}
|\chi(s)|\le e^{-s^2/250}.
\end{equation}
Here $\pi^2<10$ and $1+\lambda\le2$ suffice.

Apply \eqref{cubic:eq:8} and Lemma~\ref{lem:cubic-clock}, using
$\E |Z|^3<2$.  The result is the global Gaussian comparison
\begin{equation}\label{cubic:eq:11}
|\chi(s)-e^{-s^2/2}|
 \le256\nu^{-1/3}s^2+3000\nu^{-1/2}|s|^3.
\end{equation}
Indeed, its quadratic coefficient is at most
$250^{2/3}(1+\lambda)^{8/3}\nu^{-1/3}
\le4000^{2/3}\nu^{-1/3}<256\nu^{-1/3}$,
and its cubic coefficient is at most
$125\sqrt2(1+\lambda)^4\nu^{-1/2}
\le2000\sqrt2\nu^{-1/2}<3000\nu^{-1/2}$.

At integer mean, Fourier inversion gives
\[
 \sigma^3(2p_k-p_{k-1}-p_{k+1})
 =\frac1{2\pi}\int_{-\pi\sigma}^{\pi\sigma}
       2\sigma^2(1-\cos(s/\sigma))\operatorname{Re}\chi(s)\,ds.
\]
Use $0\le2\sigma^2(1-\cos(s/\sigma))\le s^2$ and its difference
from $s^2$ at most $s^4/(12\sigma^2)$.
For any $0<T<\pi\sigma$, splitting at $T$, applying \eqref{cubic:eq:10}--\eqref{cubic:eq:11},
and integrating the Gaussian tails gives
\begin{equation}\label{cubic:eq:12}\begin{aligned}
 &\left|\sigma^3(2p_k-p_{k-1}-p_{k+1})
                    -\frac1{\sqrt{2\pi}}\right|\\
 &\quad\le\frac{256}{15}T^5\nu^{-1/3}
       +\frac{500}{3}T^6\nu^{-1/2}+\frac4\nu
       +\frac{250}{3}\left(T+\frac{125}{T}\right)e^{-T^2/250}.
\end{aligned}\end{equation}
The multiplier error is at most $1/(8\sigma^2)\le4/\nu$.

Choose $T=56$.  If $\nu\ge2^{104}$, then
$\nu^{1/3}>3\cdot2^{33}$ and $\sqrt\nu\ge2^{52}$.
The variance lower bound $\sigma^2\ge\nu/32$ ensures $T<\pi\sigma$.
Also $56^2/250>18\log2$, using $\log2<25/36$.
Thus the right side of \eqref{cubic:eq:12} is less than
\begin{equation}\label{cubic:eq:13}
\frac{16807}{46080}
 +\frac{58824500}{3\cdot2^{34}}
 +2^{-102}
 +\frac{815250}{168\cdot2^{18}}
 <\frac{39}{100}<\frac1{\sqrt{2\pi}}.
\end{equation}
Both comparisons are rational after using $\pi<22/7$ for the last one.
For reference, the elementary bounds $11/16<\log2<25/36$ follow
from the series $\log2=2\sum_{j\ge0}(1/3)^{2j+1}/(2j+1)$,
retaining its first two terms and bounding the rest geometrically.
This proves the stated error bound and makes the central second
difference strictly positive.
The arithmetic--geometric mean inequality then gives $p_k^2>p_{k-1}p_{k+1}$;
the activity factors cancel in the polynomial determinant.
\end{proof}

\subsection{Scope and numerical reserve}
The expectations in the extension counts are uniform at fixed rank. The
response parameters $q_v$ are descendant probabilities with a vacant parent,
not whole-forest marginals. The martingale increments are revealed with
parents before children, and the variance-clock crossing in
Lemma~\ref{lem:martingale-cf-clock} is predictable. Fourier inversion is
applied to the span-one integer-valued size at an integral mean. Finally,
the prefix assembly fixes $\varepsilon=1/5$ before selecting its threshold.

The exact endpoint majorant in \eqref{cubic:eq:13} is
\[
 E_* = \frac{613959936395331007118784021397819}
 {1597239756287569045885846038773760}<\frac{39}{100}<(2\pi)^{-1/2}.
\]
Its margin to the stated error bound is $0.39-E_*=0.00561191175059\ldots$.
Any change to a norm or clock bound must be propagated through the
normalized Fourier estimate before being compared with this margin.
The accompanying rational checks verify the displayed arithmetic; the
uniform bounds in degree, activity, and order are proved above.

\section{Uniform completion and barrier consequences}
\label{sec:uniformcompletion}
We now complete the prefix and then deduce the barrier theorem.
Only the estimates proved in Sections~\ref{sec:smallrank}--\ref{sec:centralcurvature},
the fixed-density input, and the classical decreasing tail are used.
\input{figures/prefix_cover.tex}

\begin{proof}[Proof of Theorem~\ref{thm:main}]
Let $K'=2^{104}$, and let $N_{1/5}$ be an order threshold for the
fixed-density interval of \cite[Theorem~1.2]{FLNYZ}. Take
\[
 N\ge\max\{N_{1/5},\,2(K')^2\}
       =\max\{N_{1/5},\,2^{209}\}.
\]
Fix a forest of order $n\ge N$ and an integer
$1\le j\le\lfloor17\alpha(F)/25\rfloor$.
If $j<K'$, then $n\ge2j^2-4$, so Lemma~\ref{lem:extension} gives
$\Delta_j>0$. If $K'\le j<n/5$, then
$j<\eta_F(1)$ by Lemma~\ref{lem:mean}. The map $\eta_F(\lambda)$
is continuous and strictly increasing from zero, since
$\mathrm d\eta_F/\mathrm d\log\lambda=\operatorname{Var}_\lambda X>0$.
Hence there is a unique $0<\lambda<1$ with $\eta_F(\lambda)=j$.
The same mean bound gives $n\lambda\ge j\ge2^{104}$, and
Theorem~\ref{central:thm:central} gives $\Delta_j>0$.
Finally, if $j\ge n/5$, apply the cited fixed-density theorem.
These three cases cover every required integer rank.
\end{proof}

The quantity $2^{209}$ is the explicit lower-rank cost of this assembly;
$N_{1/5}$ remains the threshold supplied by the cited fixed-density theorem.
Thus this is an existential threshold for the full prefix, not an explicit
numerical value for it. The quantitative supplement supplies a second
assembly and separately qualified explicit consequences.

\begin{corollary}[Eventual separation of log-concavity defects]
For every sufficiently large forest, every internal log-concavity failure occurs at
index $j>17\alpha(F)/25$. In particular, all such defects occur in the
known decreasing terminal segment.
\end{corollary}
\begin{proof}
The first assertion is Theorem~\ref{thm:main}. The decreasing-tail bound of
Levit--Mandrescu starts at $\lceil(2\alpha-1)/3\rceil$ for forests
\cite{LM}; see also \cite[Proposition~8.2]{FLNYZ}.
Since $17/25>2/3$, a defect above $17\alpha/25$ lies in this segment.
Moreover, it also lies at or after
$\lceil\alpha(n-1)/(n+\alpha)\rceil$, because $\alpha\ge n/2$ gives
$\alpha(n-1)/(n+\alpha)<2\alpha/3$; this is the cutoff in
\cite[Theorem~3]{BG}.
\end{proof}

\begin{proof}[Proof of Theorem~\ref{cor:barrier}]
Put
\[
 C=\left\lceil\frac{2\alpha(F)-1}{3}\right\rceil.
\]
The classical tail bound gives $i_C\ge i_{C+1}\ge\cdots\ge i_\alpha$
\cite{LM}; see also \cite[Proposition~8.2]{FLNYZ}.
If $i_{j+1}>i_j$ at an internal index, then $j<C$, so
\[
 j\le C-1<\frac{2\alpha(F)-1}{3}<\frac{17\alpha(F)}{25}.
\]
Since $j$ is an integer, Theorem~\ref{thm:main} gives $\Delta_j>0$.

For $0<t\le1$, put $g_j=t^ji_j$. An ascent $g_{j+1}>g_j$ implies
$i_{j+1}/i_j>1/t\ge1$, and therefore $\Delta_j>0$ by the first part.
The tilted determinant is
\[
 g_j^2-g_{j-1}g_{j+1}=t^{2j}\Delta_j>0.
\]
All coefficients on the support remain positive because $t>0$.
The barrier condition implies unimodality, as recalled in the introduction.
\end{proof}
The same order threshold proves the prefix and barrier conclusions.
The optional LVZ-dependent explicit bound belongs to the quantitative
supplement and is not used in this proof.

\paragraph{Why ordinary log-concavity is not claimed.}
Kadrawi and Levit \cite{KL26} give infinite families of trees whose
independence polynomials are not log-concave, extending the order-26
counterexamples of Kadrawi, Levit, Yosef, and Mizrachi \cite{KLYM}.
For one of their basic 26-vertex trees,
\[
 (i_{12},i_{13},i_{14})=(2979,51,1),\qquad
 \Delta_{13}=51^2-2979=-378.
\]
This defective triple is strictly decreasing and therefore does not
violate the BL implication. The infinite families show that ordinary
log-concavity cannot replace BL in a theorem valid for all sufficiently
large trees.

Galvin \cite{GalvinFailures} shows that failures need not remain a
bounded distance below the independence number. Bautista-Ramos
\cite{BRMultiple2025} constructs trees with arbitrarily many failing
indices, while Bautista-Ramos, Guill\'en-Galv\'an, and G\'omez-Salgado
\cite{BRGG2026} obtain infinite families with two and three consecutive
failures; their four- and five-consecutive-failure collections are finite.
Bautista-Ramos \cite[Corollary~3]{BautistaSpherical} subsequently constructs
subcubic trees with arbitrarily many ordinary log-concavity failures at
sufficiently large parameters. Taking those parameters also beyond the
threshold in Theorem~\ref{cor:barrier} shows that strictly BL subcubic
trees can have arbitrarily many such failures across the family. The
locations certified by that construction are discussed in the optional
quantitative supplement.

\paragraph{Perfect matchings do not prevent severe tail failures.}
Fan \cite[Theorems~1.2 and 1.3]{Fan2026} constructs trees with a perfect
matching, so that $|V(T)|=2\alpha(T)$, with a failure at any prescribed
positive distance below $\alpha(T)$, or with simultaneous failures at
$\alpha-1,\alpha-3,\ldots,\alpha-(2t-1)$. The normalized ratios at the
specified failures tend to zero as the scale parameter grows. These
examples refute the failure-location implication of Ramos and Sun
\cite[Conjecture 4.2]{RamosSun2025}, not the barrier conjecture.

The specified failures lie on strictly decreasing triples. Indeed,
for a forest with $\alpha=m$, extension counting gives
$(j+1)i_{j+1}\le2(m-j)i_j$: the residual after an independent $j$-set
has independence number at most $m-j$ and at most $2(m-j)$ vertices.
At $j=m-r$ this implies
\[
 \frac{i_j}{i_{j-1}}\le\frac{2(r+1)}{m-r},\qquad
 \frac{i_{j+1}}{i_j}\le\frac{2r}{m-r+1}.
\]
Both bounds are below one when $m>3r+2$. In Fan's prescribed-distance
family, $r=q-1$ and $m=(3q+1)p+q+1\ge4r+6$ for $p\ge1$.
In his simultaneous family, $r\le2t-1$ and
$m=t(3+7p)\ge10t$, so the same conclusion holds. This controls those
specified triples; it does not establish BL at every rank for every
parameter in either family.

Combining Fan's simultaneous-failure theorem with
Theorem~\ref{cor:barrier} gives the following further consequence:
\emph{for every integer $t\ge2$ and every $\varepsilon>0$, there exists
a strictly BL tree with a perfect matching and distinct indices
$j_s=\alpha-(2s+1)$, $0\le s<t$, such that}
\[
 0<\frac{i_{j_s}^2}{i_{j_s-1}i_{j_s+1}}<\min\{1,\varepsilon\}
 \qquad(0\le s<t).
\]
To prove it, fix $t$ and choose Fan's scale large enough both for its
finite collection of ratios to be below $\min\{1,\varepsilon\}$ and
for the tree order to exceed the threshold of Theorem~\ref{cor:barrier}.
This consequence uses Fan's construction as an additional input; none
of the preceding theorems depends on it. A perfect matching is not
well-coveredness: for a nontrivial well-covered tree its matching edges
must all be pendant, as proved in Section~\ref{app:corona}.

\subsection{Palindromic independence polynomials}\label{subsec:palindromic}
A polynomial $a(x)=\sum_{j=0}^d a_jx^j$ is \emph{palindromic} if
$a_j=a_{d-j}$ for $0\le j\le d$, or equivalently $x^d a(1/x)=a(x)$.
For an independence polynomial, $d=\alpha(G)$, not $|V(G)|$.
This is coefficient symmetry, not symmetry of a drawing or of rooted
branching levels; see \cite{LMPal2007,LMSym2011,HKV}.

\begin{lemma}[Reflection and the barrier condition]\label{lem:palindromic-bl}
For a positive palindromic sequence $a_0,\ldots,a_d$,
\[
 \Delta_j:=a_j^2-a_{j-1}a_{j+1}=\Delta_{d-j}\qquad(1\le j<d).
\]
The sequence is BL if and only if it is log-concave. Strict log-concavity
through $\lfloor d/2\rfloor$ implies strict log-concavity at every internal
index.
\end{lemma}
\begin{proof}
Reflection gives the determinant identity. Suppose the sequence is BL
and $\Delta_j<0$. Then $a_{j+1}\le a_j$, and
$a_{j-1}a_{j+1}>a_j^2$ implies $a_{j-1}>a_j$.
At the reflected center $d-j$ this is a strict ascent with negative
determinant, contradicting BL. The converse is immediate.
For the last assertion, reflect an internal center $j$ to
$\min\{j,d-j\}\le\lfloor d/2\rfloor$.
\end{proof}

\begin{corollary}[Large forests with palindromic independence polynomials]
\label{cor:palindromic-large}
Let $N$ be the common threshold in Theorems~\ref{cor:barrier}--\ref{thm:main}.
If $F$ is a forest of order at least $N$ and $I(F;x)$ is palindromic, then
$I(F;x)$ is strictly log-concave at every internal index. Writing
$d=\alpha(F)$, its modes are exactly
\[
 \{d/2\}\quad\text{if $d$ is even},\qquad
 \{(d-1)/2,(d+1)/2\}\quad\text{if $d$ is odd}.
\]
Every common positive coefficient tilt remains strictly log-concave.
\end{corollary}
\begin{proof}
Theorem~\ref{thm:main} covers every center through $\lfloor d/2\rfloor$,
since $17/25>1/2$. Reflect by Lemma~\ref{lem:palindromic-bl}.
Strictly decreasing adjacent ratios and coefficient symmetry give the
stated unweighted modes. A common tilt multiplies $\Delta_j$ by $t^{2j}>0$.
\end{proof}
The strict conclusion uses the protected half-prefix, not strict BL alone:
a positive symmetric sequence may be strictly BL and have a longer modal
plateau. The existence constructions of Hibi, Kara, and Vien
\cite[Theorem~0.1]{HKV} and the present result address different quantifiers:
they construct symmetric unimodal trees at the stated orders and degrees,
whereas this corollary applies to every sufficiently large forest with
coefficient symmetry.

\section{The corona transfer and well-covered forests}\label{app:corona}

The application to well-covered forests is proved here in full.
Log-concavity of cumulative coefficients is log-concavity, on the
integer lattice, of the distribution function of the normalized
coefficient sequence; for
general distribution-function results, see Finner and Roters
\cite{FinnerRoters1993}. Figure~\ref{fig:corona-decomposition}
illustrates the distinction between the base order, the nonisolated order,
and additional isolated vertices.
\input{figures/corona_decomposition.tex}

The sequence transfer also appears in the signed-ascent companion \cite[Appendix on the leaf-corona transform]{LevitCorona}.
We reproduce its proof so that the consequence below has no proof dependency
on that companion's other results. The leaf-corona identity and the
well-covered-tree characterization are classical
\cite[Theorem~3 and the paragraph preceding Theorem~6]{LM2004}; see also \cite{FHN,BC2018}.
Levit and Mandrescu \cite{LM1999} also characterized well-covered trees
as well-covered spiders or internal edge-joins of well-covered spiders.
Here a spider is a tree with at most one vertex of degree at least three,
and an internal edge-join adds an edge between vertices of degree at
least two in otherwise disjoint trees. This structural antecedent is distinct from
the corona-polynomial identity; the corona representation needed below
is proved in full.
The term \emph{corona} here means the graph operation $H\circ K_1$, not the
union of maximum independent sets sometimes denoted $\operatorname{corona}(H)$.

\begin{lemma}[BL gives log-concave cumulative coefficients]\label{lem:corona-cumulative}
Let $a_0,\ldots,a_d>0$ be BL, and put $a_k=0$ for $k>d$ and
$A_k=\sum_{i=0}^k a_i$. The sequence $(A_k)_{k\ge0}$ is log-concave,
strictly at $1\le k\le d$.
\end{lemma}
\begin{proof}
Its determinant is $A_k^2-A_{k-1}A_{k+1}=a_kA_k-a_{k+1}A_{k-1}$.
If $a_{k+1}\le a_k$ and $k\le d$, this is at least $a_k^2>0$.
Otherwise, the BL implication, applied backward from the actual ascent,
gives $a_1/a_0\ge\cdots\ge a_{k+1}/a_k>1$. Hence
$a_{k+1}a_i\le a_ka_{i+1}$ for $0\le i<k$, and
$a_{k+1}A_{k-1}\le a_k(A_k-a_0)<a_kA_k$.
For $k>d$ the cumulative coefficients are constant.
\end{proof}

Classical binomial-transform preservation starts from a log-concave
input sequence; see, for example, Zhu
\cite[Theorems~1.1--1.2]{BZhuTransform2016}. The hypothesis below instead
concerns its cumulative coefficients. We give a direct proof, including
the strict terminal determinant, rather than infer the result from an
ordinary binomial-transform statement.

\begin{theorem}[Cumulative-coefficient corona transfer]\label{thm:corona-transfer}
Let $a(x)=\sum_{k=0}^d a_kx^k$, where $a_0>0$ and $a_k\ge0$.
If its cumulative coefficients are log-concave, then, for every integer
$m\ge d$, the polynomial
\[
 \mathcal C_m[a](x)=(1+x)^m a\!\left(\frac{x}{1+x}\right)
             =\sum_{k=0}^d a_kx^k(1+x)^{m-k}
\]
has positive coefficients in degrees $0,\ldots,m$ and is strictly
log-concave at all its internal indices. In particular, this holds for
every positive BL input.
\end{theorem}
\begin{proof}
We first prove a finite tail-moment inequality. Let $c_0,\ldots,c_m\ge0$,
$c_m>0$, $m\ge1$, with log-concave tail sums. Normalize them to the law
of a random variable $R$, and write $S_j=\Prob(R\ge j)$.
Then $S_0=1$, $S_m>0$, and $S_j=0$ for $j>m$.
Decreasing ratios of the positive tail sums give
$S_{i+j}\le S_iS_j$. Consequently,
\[
 \E\binom R2=\sum_{i,j\ge1}S_{i+j}
       <\sum_{i,j\ge1}S_iS_j=(\E R)^2.
\]
The first identity counts the $\binom r2$ positive pairs with $i+j\le r$.
Strictness follows from the pair $i=j=m$, for which $S_{2m}=0<S_m^2$.
Thus
\begin{equation}\label{eq:corona-tail-moment}
 \left(\sum r c_r\right)^2>
       \left(\sum c_r\right)\left(\sum\binom r2 c_r\right).
\end{equation}
Pad $a$ with zeros through degree $m$ and induct on $m$. The cases
$m\le1$ have no internal indices. Truncating the input preserves
cumulative log-concavity, and
\[
 \mathcal C_m[a]=(1+x)\mathcal C_{m-1}[a_0,\ldots,a_{m-1}]+a_mx^m.
\]
The term $a_0(1+x)^m$ ensures positive coefficients throughout.
Multiplication by $1+x$ preserves strict log-concavity: for
$q_j=b_j+b_{j-1}$, with $b$ zero outside its positive interval support,
\[
 \Delta_j(q)=\Delta_j(b)+\Delta_{j-1}(b)
                 +b_jb_{j-1}-b_{j-2}b_{j+1}.
\]
The last term is nonnegative by decreasing adjacent ratios, including
the zero endpoints. This proves all needed determinants except the
last one, since adding $a_mx^m$ changes no earlier internal triple.
Writing $B_j=[x^j]\mathcal C_m[a]$, the last three coefficients are
\[
 B_m=\sum a_k,\qquad B_{m-1}=\sum(m-k)a_k,\qquad
 B_{m-2}=\sum\binom{m-k}{2}a_k.
\]
Apply \eqref{eq:corona-tail-moment} to $c_r=a_{m-r}$.
Its tail sums are reversed input prefix sums and $c_m=a_0>0$.
This gives $B_{m-1}^2>B_{m-2}B_m$ and completes the induction.
The BL case follows from Lemma~\ref{lem:corona-cumulative}.
\end{proof}

\begin{lemma}[Classical well-covered-tree representation]\label{lem:wc-representation}
Every well-covered tree $T\ne K_1$ has the form $H\circ K_1$ for a tree
$H$, with $|V(T)|=2|V(H)|$. Conversely, every such corona is well-covered.
\end{lemma}
\begin{proof}
The characterization is classical \cite{FHN,LM2004}; we give its tree
proof including the small cases. In a nontrivial tree, both bipartition
classes are maximal independent sets. Well-coveredness makes both have
size $\alpha(T)$. If they are $L,R$, then
$S\cup(R\setminus N(S))$ is independent for $S\subseteq L$, giving
$|N(S)|\ge|S|$. Hall's theorem supplies a perfect matching.
Every matching edge must be pendant: if $uv$ is not, choose neighbors
$u'\ne v$ of $u$ and $v'\ne u$ of $v$. Acyclicity makes $u',v'$
distinct and nonadjacent. A maximal independent set extending
$\{u',v'\}$ misses both $u,v$ and has at most one vertex in every other
matching edge. Its size is at most $\alpha(T)-1$, a contradiction.
For $T=K_2$, use $H=K_1$. Otherwise, every nonleaf has exactly one leaf
neighbor, and the induced subgraph of nonleaves is connected and acyclic.
It is $H$. Conversely, every maximal independent set in $H\circ K_1$
contains exactly one vertex from each old-vertex/private-leaf pair.
\end{proof}

\begin{proof}[Proof of Theorem~\ref{cor:eventual-wc}]
A maximal independent set in a disjoint union restricts to a maximal
independent set in each component, and any such componentwise choices
combine to a maximal independent set. Thus a forest is well-covered
if and only if each component is well-covered. Apply
Lemma~\ref{lem:wc-representation} to every nontrivial component to write
\[
 F=(H\circ K_1)\sqcup sK_1,\qquad s=z(F),\qquad m=|V(H)|.
\]
The $m$ pendant edges form a perfect matching of the nonisolated part,
and no matching can have more than $m$ edges there. Hence
$\mu(F)=m$ and $|V(F)|-z(F)=2m$, proving the equivalence of the hypotheses.
For any graph $H$ of order $m$, conditioning on its selected independent
set gives the classical identity \cite[Theorem~3]{LM2004}
\begin{equation}\label{eq:corona-identity}
 I(H\circ K_1;x)=(1+x)^m I\!\left(H;\frac{x}{1+x}\right).
\end{equation}
Consequently,
\[
 I(F;x)=(1+x)^{m+s}I\!\left(H;\frac{x}{1+x}\right).
\]
If $\mu(F)\ge N$, Theorem~\ref{cor:barrier} gives BL for the
independence polynomial of the whole base forest $H$.
Apply Theorem~\ref{thm:corona-transfer} to $a(x)=I(H;x)$ with
transform order $m+s\ge\alpha(H)$ to obtain full strict log-concavity
of $I(F;x)$. Multiplying coefficients by $t^k$ multiplies their
determinants by $t^{2k}$, so every common $t>0$ is allowed. When there
are no isolated vertices, $|V(F)|=2\mu(F)$, giving the stated tree
and isolate-free-forest specializations.
\end{proof}

\paragraph{Isolated vertices and scope.}
The matching-number hypothesis concerns the nonisolated part, not total
order enlarged by isolated vertices. Adding isolated vertices leaves
$\mu(F)$ unchanged. For example, $K_2\sqcup sK_1$ has matching number
one at arbitrarily large orders; it is not a counterexample to
log-concavity, but it does not force the base above $N$.
This argument does not exclude a different large-total-order theorem
for well-covered forests. For any family of base trees known to be BL, the single-leaf corona
is strictly LC at every order. This applies, for example, to the binary
spherical family under the hypotheses and proof inputs of
\cite{LevitSpherical}. No generic
BL product closure is used.

\paragraph{Relation to prior results.}
Levit--Mandrescu \cite[Theorems~4, 5, and 7]{LM2004} proved LC for
several corona classes, including the coronas of stars and paths.
Brown--Cameron \cite[Theorem~2.4]{BC2018} prove strict LC after
sufficiently many iterated leafy extensions of each fixed graph; the
required number depends on the roots of that graph's polynomial.
Here one extension suffices uniformly over all sufficiently large
\emph{forest bases}. The published clique-corona theorem in
\cite{HLMP2026} instead guarantees LC when the attached clique size is
sufficiently large relative to the base. It is not used as an all-base,
fixed-$K_1$ theorem. The cumulative-coefficient transfer and the eventual forest theorem are
the two ingredients in the present application. The conclusion does not
assert log-concavity for well-covered forests of every order.

\section{Conclusion and open problems}\label{sec:conclusion}
Every sufficiently large forest is strictly barrier log-concave, and every
sufficiently large well-covered tree is strictly log-concave. The stronger
well-covered-forest statement uses matching number rather than total order
when isolated vertices are allowed. These results follow from a uniform
curvature theorem, completion of the known forest prefix, and the corona
transfer. Coefficient reflection also gives full strict log-concavity and
exact central modes for sufficiently large palindromic forests.
The principal question here is whether the order restriction can be removed
from the barrier conclusion. The all-order unweighted unimodality
statements in \cite{ZhangLi2026,LVZ} do not by themselves answer that
stronger question.

\begin{conjecture}[Barrier log-concavity for trees]\label{conj:tree-bl}
For every finite tree $T$ and every $1\le k<\alpha(T)$,
\[
 i_{k+1}(T)>i_k(T)\quad\Longrightarrow\quad
 i_k(T)^2\ge i_{k-1}(T)i_{k+1}(T).
\]
\end{conjecture}
Equivalently, the independence polynomial of every tree is unimodal under
\emph{every common downward activity}: $I(T;tx)$ is unimodal for each
$0<t\le1$. For a positive sequence, this equivalence follows directly
from its adjacent ratios. A barrier failure has ratios
$r_{k+1}>\max\{1,r_k\}$; choosing
$1/r_{k+1}<t<\min\{1,1/r_k\}$ creates a strict valley. Conversely,
a downward tilt preserves the barrier implication and hence unimodality.
Ordinary log-concavity counterexamples do not refute
Conjecture~\ref{conj:tree-bl}: the conjecture allows curvature failures on
the decreasing portion, but excludes them at an ascending center.
On the subclass of trees with palindromic independence polynomials,
Lemma~\ref{lem:palindromic-bl} makes the BL assertion equivalent to
ordinary log-concavity. Corollary~\ref{cor:palindromic-large} proves the
strict eventual version without a new supplementary premise.

\begin{conjecture}[Barrier log-concavity for forests]\label{conj:forest-bl}
For every finite forest $F$, the polynomial $I(F;tx)$ is unimodal for
every $0<t\le1$; equivalently, $I(F;x)$ is barrier log-concave.
\end{conjecture}
The forest assertion includes the tree assertion. However, products of
barrier log-concave sequences cannot be assumed to retain that property,
so a tree theorem would require an additional argument to cover
disconnected forests.
The present theorem proves the forest assertion once the order is large
enough. The sharp-prefix companion~\cite{LevitDraft} supplies an
all-order guarantee at the first twelve coefficient indices.

\begin{problem}[Prescribed-prefix curvature]
With $q(F)=\lceil(2\alpha(F)-1)/3\rceil$, determine whether every forest
satisfies $\Delta_k(F)\ge0$ at each internal center $k<q(F)$.
\end{problem}
This asks for more than barrier log-concavity: the prescribed interval may
contain descending centers after the actual mode. The present paper
establishes the condition for sufficiently large forests.

\begin{problem}[Well-covered forests]
Does every well-covered forest have a log-concave independence polynomial?
\end{problem}
Conjecture~\ref{conj:forest-bl} and Section~\ref{app:corona} would give
the stronger strict conclusion at all internal indices, including after
adjoining isolated vertices. The matching-number threshold in
Theorem~\ref{cor:eventual-wc} is the sufficient condition proved
here; it is not asserted to be necessary.

\begin{problem}[Effective bounds and modes]
Obtain a substantially smaller explicit order threshold for the eventual
barrier theorem, and determine how accurately the first mode can be
located from structural data or a bounded number of activity moments.
\end{problem}
A useful mode algorithm must account for the cost of computing the
coefficients or moments it queries. An exact ratio-crossing characterization
alone is not a running-time bound. A structural proof of
Conjecture~\ref{conj:tree-bl} or~\ref{conj:forest-bl}, independent of
asymptotic local-limit estimates, would address both the shape question
and the search for more intrinsic mode information.

\subsection{Barrier log-concavity beyond forests}\label{subsec:beyond-forests}
Barrier log-concavity may also strengthen unimodality theorems outside
forest enumeration. The general background includes the survey of Stanley
\cite{Stanley1989}. The substantive setting is a natural nonsymmetric rank sequence for which ordinary log-concavity actually fails. The
following elementary examples use the original cardinality ranks,
including the empty object, and do not depend on the asymptotic forest
theorem.

\begin{example}[Order ideals above an antichain]\label{ex:ideal-bl}
Let $P_{r,\ell}$ consist of $r\ge2$ mutually incomparable elements below
an $\ell$-element chain, where $\ell\ge1$. An order ideal is a subset
containing every predecessor of each of its elements. Ordered by
inclusion, the ideals form the distributive lattice $J(P_{r,\ell})$;
see \cite[Chapter~3]{StanleyEC1}. Counting ideals by cardinality gives
\[
 R_{r,\ell}(x)=\sum_{A\in J(P_{r,\ell})}x^{|A|}
             =(1+x)^r+x^{r+1}+\cdots+x^{r+\ell}.
\]
Indeed, an ideal either is a subset of the antichain or contains the
whole antichain followed by a nonempty initial segment of the chain.
The coefficient sequence is strictly BL: before rank $r$ the internal
triples are binomial, and from rank $r$ onward the coefficients are all
one, so no new ascent occurs. Nevertheless, its determinant at rank
$r$ is $1-r<0$. For example, $R_{4,1}$ has coefficients
\[
 (1,4,6,4,1,1),\qquad \Delta_4=-3.
\]
Thus even a five-element poset gives a nonsymmetric, strictly BL ideal-rank
sequence that is not log-concave. The ratio at the failure is $1/r$,
which can be arbitrarily small in this family.
\end{example}

The example gives a structural preservation rule. If the ideal-rank
sequence of a finite nonempty poset $P$ is BL, adjoining a chain above
every element of $P$ preserves BL: the rank polynomial simply acquires
a terminal string of ones. Moreover, if $P$ has $b\ge2$ maximal
elements, the previous terminal
rank becomes an internal center with
determinant $1-b<0$. This follows because the coefficient of rank
$|P|-1$ is exactly the number of maximal elements. The statement also
holds with strict BL in place of BL. It is a particular operation on
posets, not a closure assertion for arbitrary sums or products of BL
polynomials.

\begin{example}[Clique counts with pendant vertices]\label{ex:clique-bl}
Let $G_s$ be obtained from $K_4$ by attaching $s\ge0$ pendant vertices
to one fixed vertex. Its clique polynomial counts complete vertex
subsets, including the empty set once; see \cite{HoedeLi1994}.
Every new vertex creates just its singleton clique and its incident edge.
Consequently,
\[
 C(G_s;x)=(1+x)^4+s x(1+x)
        =1+(s+4)x+(s+6)x^2+4x^3+x^4.
\]
Writing $c_k=[x^k]C(G_s;x)$, the three internal determinants are
\[
 \Delta_1=s^2+7s+10,\qquad
 \Delta_2=s^2+8s+20,\qquad
 \Delta_3=10-s.
\]
The first two are positive, while $c_4=1<c_3=4$. Hence the clique-rank
sequence is strictly BL for every $s\ge0$, but it is log-concave if and
only if $s\le10$. At $s=11$ the fifteen-vertex graph has sequence
\[
 (1,15,17,4,1),\qquad \Delta_3=-1.
\]
For $s>10$ the defective triple is strictly decreasing, and
\[
 \frac{c_3^2}{c_2c_4}=\frac{16}{s+6}\longrightarrow0.
\]
Thus arbitrarily severe log-concavity failure is possible in this
strictly BL family even with maximum clique size fixed at four.
These are also the face counts of the clique complex, indexed by
cardinality; they are not the independence coefficients of $G_s$.
\end{example}

Both examples remain unimodal under every common downward tilt,
although full log-concavity fails. Related independence-polynomial examples occur outside forests:
Levit and Kadrawi \cite{LevitKadrawi2026} construct unicyclic families
that are unimodal but not LC. Unimodality alone does not establish BL
for those families. Such a strengthening is not automatic, as the
following classical family demonstrates.

\paragraph{The Gaussian warning.}
Gaussian binomial polynomials are classically symmetric and unimodal;
see \cite[Introduction]{PakPanova2013}. Nevertheless,
\[
 \binom{4}{2}_{\!x}=1+x+2x^2+x^3+x^4
\]
has coefficients $(1,1,2,1,1)$ and fails BL at its first internal center:
$1^2-1\cdot2=-1$ on an ascent. At $t=3/4$, its downward tilt starts
with the strict valley $1>3/4<9/8$. More generally,
$\binom{a+b}{a}_{x}$ starts with $1,1,2$ whenever $a,b\ge2$, because
there are two partitions of two that fit in the $a$-by-$b$
rectangle.
Thus this whole classical unimodal family fails BL, not just one small
member. Gaussian coefficients also count order ideals of a product of
two chains, so distributivity alone cannot imply the barrier property.
For positive palindromic sequences, Lemma~\ref{lem:palindromic-bl}
already shows that BL is equivalent to ordinary log-concavity. The most
distinctive applications of the intermediate property should therefore
be sought among nonsymmetric sequences.

\begin{problem}[Barrier strengthening of unimodality]\label{prob:barrier-strengthening}
Identify further natural nonsymmetric combinatorial rank sequences that
are barrier log-concave but not log-concave, and determine structural
conditions under which a known unimodality theorem extends to every
uniform downward tilt.
\end{problem}
The terminal-extension rule supplies one mechanism. The broader problem
is to control several defect ranks or obtain a structural preservation
principle beyond appending a nonincreasing tail.

\paragraph{Ancillary verification and optional quantitative material.}
The main article contains the proofs of all its stated results, with their
external inputs identified where used. The optional quantitative supplement
\cite{QuantitativeSupplement} supplies sharper low-activity estimates, an
explicit through-mean bound, an all-order small-tilt interval, and separately
conditional explicit barrier and well-covered-forest cutoffs. Its explicit
symmetric-forest bound and $\gamma$-positivity comparisons are refinements,
not premises of Corollary~\ref{cor:palindromic-large}.
The accompanying code checks 51 exact rational guards and finite
weighted-tree, martingale, and corona identities. From the extracted
source directory, run
\begin{verbatim}
python3 anc/run_checks.py --output-dir ../p1-checks
\end{verbatim}
The elementary rank examples in Section~\ref{subsec:beyond-forests} have
a separate exact checker:
\begin{verbatim}
python3 anc/beyond_forests/verify_examples.py
\end{verbatim}
These finite checks do not replace the analytic proofs. The external
LVZ computational and formal developments are not run by this command.
The exact certificates of the sharp-prefix companion \cite{LevitDraft}
are not premises of the main theorems here.

\section*{Use of AI tools}
ChatGPT was used for mathematical exploration and proof development,
code preparation and checking, literature comparison, and language
editing. The author is responsible for the mathematical content,
attribution, and computational materials. This statement does not
assert formal verification.

\end{document}


\maketitle
\begin{abstract}
This supplement provides the quantitative refinements accompanying
\emph{Barrier Log-Concavity of Large Forests}. It proves sharper
low-activity curvature with starting rank $31\,500\,000$, an explicit
through-mean cutoff, and an all-order interval of small coefficient tilts.
A separate section derives an explicit barrier cutoff from a precisely
specified additional LVZ input. The corresponding well-covered-forest
and symmetric-polynomial consequences, complete estimates, and numerical
margins are included. The principal eventual-BL and well-covered-forest
theorems in the main article do not depend on the refinements here.
\end{abstract}

\section{Statements, notation, and dependencies}
References prefixed \textup{M} refer to the accompanying main article
\cite{MainPaper}; references prefixed \textup{S} are local to this supplement.
For example, Theorem~\MainRef{cor:barrier} is its eventual strict-BL theorem.
This supplement uses results of the main article; its own additional
estimates are not premises of the main article. The main article proves the elementary extension and
mean bounds, bipartition damping, uniform central curvature, completed
prefix, and full corona transfer used below. All additional arguments
are supplied here; Section~\ref{sec:lvz} explicitly identifies its external
input and is not needed for the other refinements.

Zhang--Li \cite[Theorem~1.1]{ZhangLi2026} and Li--Vallier--Zhang
\cite[Theorem~1.1]{LVZ} both state all-order unweighted unimodality;
the latter gives a second proof building on the former. Neither headline
theorem is used here: the extra input in Section~\ref{sec:lvz} is the
fixed LVZ weighted inequality, not unweighted unimodality.

We use the main notation
\[
 Z_F(x)=I(F;x),\quad X=|I|,\quad
 p_j=i_j\lambda^j/Z_F(\lambda),\quad
 \eta=\E_\lambda X,\quad \sigma^2=\Var_\lambda X,\quad \nu=n\lambda.
\]
The physical activity is common unless stated otherwise; forests are
nonempty whenever variance is a divisor. The matching number is
$\mu(F)$, not the activity mass $\nu$. Set
\[
 \lambda_* = 1/3125,\qquad K=31\,500\,000.
\]
\begin{theorem}[Explicit low-fugacity curvature]\label{low:thm:main}
For every forest, if $0<\lambda\le\lambda_*$, $\nu\ge K$, and
$\eta=k$ is an integer, then
\[
 \sigma^3(2p_k-p_{k-1}-p_{k+1})>\frac1{2000}.
\]
Consequently $\Delta_j(F)>0$ at every integer rank
$K\le j\le n/3128$.
\end{theorem}

The even comparison in Section~\ref{low:sec:even} is the key local
estimate. In the real part of a product of characteristic functions,
odd errors contribute in pairs. After centroid conditioning, the
leading curvature error is $O(\nu^{-1/4})$ at block size
$\lceil(293/500)\nu^{3/8}\rceil$. We explicitly control the correlation between
the random conditional mean and the conditional characteristic function.
The variance estimate $\sigma^2\ge n\lambda/(1+3\lambda)^2$ supplies
a lower constant close to one at low activity. Together with the even
comparison, it gives the displayed rank threshold by rational inequalities.

\begin{corollary}[An explicit cutoff through the activity-one mean]
\label{central:cor:throughmean}
Set $N_0=3125\cdot2^{104}<2^{116}$. Every forest of order $n\ge N_0$
is strictly log-concave at every internal integer rank
$1\le j\le\eta_F(1)$.
\end{corollary}

The corollary combines Theorem~\MainRef{central:thm:central} with
Theorem~\ref{low:thm:main} and the elementary quadratic bound proved in
Section~\MainRef{sec:smallrank}; the assembly appears in
Section~\ref{sec:quantitativecompletion}. None of these arguments uses an all-forest
unimodality theorem or a finite enumeration of forests. Since
$\eta_F(1)\ge n/4$, the corollary includes all internal ranks $j\le n/4$.
The sharper low-activity estimate improves the explicit order cutoff
obtained from the central theorem alone.

\begin{corollary}[A uniform interval of downward tilts]\label{cor:smalltilt}
Let $B=2K-5=62999995$. For every forest and every $0<t\le1/B$,
\[
 t^{j+1}i_{j+1}>t^ji_j\quad\Longrightarrow\quad\Delta_j>0
 \qquad(1\le j<\alpha(F)).
\]
Thus $Z_F(tz)$ is strictly barrier log-concave and unimodal throughout
this interval.
\end{corollary}

These statements do not require External input~\ref{input:lvz}.
The optional explicit barrier conclusion is Corollary~\ref{cor:explicit-barrier};
it is not an explicit cutoff for the full $17\alpha/25$ prefix.

\section{Additional variance bounds}
\begin{lemma}[A covariance bound from single-vertex conditioning]\label{lem:scorevariance}
Let $G$ be a nonempty graph with $m$ edges, and let $X$ be its hard-core
independent-set size at activity $\lambda>0$, with mean $\eta$. Then
\[
 \operatorname{Var}X\ge
 \frac{\eta^2}{(1+\lambda)\eta+2m\lambda^2}.
\]
For a forest on $n$ vertices this implies
\[
 \operatorname{Var}X\ge
 \frac{n\lambda}{1+6\lambda+15\lambda^2+18\lambda^3}.
\]
\end{lemma}
\begin{proof}
Put $p=\lambda/(1+\lambda)$, $X_v=\1_{\{v\in I\}}$,
$Y_v=\1_{\{N(v)\cap I=\varnothing\}}$, and $R_v=X_v-pY_v$.
Conditional on all indicators except $X_v$, the conditional mean of
$X_v$ is $pY_v$, so $R_v$ is conditionally centered. In particular,
$\E R_v=0$ and $\E R_v^2=(1-p)\E X_v$.
If $u\ne v$ are nonadjacent, then $R_u$ is measurable with respect to
all indicators except $X_v$, giving $\E R_uR_v=0$.
If $uv$ is an edge, the identities
$X_uX_v=X_uY_v=X_vY_u=0$ instead give
$\E R_uR_v=p^2\E Y_uY_v$.

For $R=\sum_vR_v$, it follows that
\[
 \Var R=(1-p)\eta+2p^2\sum_{uv\in E(G)}\E Y_uY_v
       \le(1-p)\eta+2mp^2.
\]
Conditional centering also gives $\E X_uR_v=0$ for $u\ne v$ and
$\E X_vR_v=(1-p)\E X_v$. Hence
$\operatorname{Cov}(X,R)=(1-p)\eta$.
Cauchy--Schwarz proves
\[
 \Var X\ge\frac{(1-p)^2\eta^2}{(1-p)\eta+2mp^2}
         =\frac{\eta^2}{(1+\lambda)\eta+2m\lambda^2}.
\]
For forests, use $m\le n$ and the independently proved mean bound
$\eta\ge n\lambda/(1+3\lambda)$ in Lemma~\MainRef{lem:mean}.
The function $z\mapsto z^2/((1+\lambda)z+2n\lambda^2)$ is increasing
for $z>0$. Substitution gives the stated forest bound.
\end{proof}

\begin{lemma}[Forest extension variance]\label{lem:extensionvariance}
For every nonempty forest and every $\lambda>0$,
\[
 \boxed{\sigma^2\ge\frac{\eta}{1+3\lambda}
                 \ge\frac{n\lambda}{(1+3\lambda)^2}.}
\]
Consequently, if $0<\lambda\le1/3125$, then
$\sigma^2>(499/500)n\lambda$.
\end{lemma}
\begin{proof}
Use the extension averages $m_r=(r+1)i_{r+1}/i_r$, with
$m_{\alpha}=0$. The proof of Lemma~\MainRef{lem:extension} gives
$m_{r+1}\ge m_r-3$ at every supported rank, including the terminal
step. Thus $h_r=m_r+3r$ is nondecreasing.
Weighted extension counting gives the exact identities
\[
 \E m_X=\frac{\eta}{\lambda},\qquad
 \E(Xm_X)=\frac{\E(X(X-1))}{\lambda},\qquad
 \operatorname{Cov}(X,m_X)=\frac{\sigma^2-\eta}{\lambda}.
\]
If $X'$ is an independent copy of $X$, monotonicity of $h$ implies
\[
 0\le\frac12\E[(X-X')(h_X-h_{X'})]
   =\operatorname{Cov}(X,h_X)
   =\frac{\sigma^2-\eta}{\lambda}+3\sigma^2.
\]
Rearranging proves the first inequality, and Lemma~\MainRef{lem:mean}
proves the second. Finally,
$(1+3/3125)^2<500/499$ proves the displayed low-activity constant.
Both cited extension lemmas are proved directly by counting and Jensen's
inequality in Section~\MainRef{sec:smallrank}; they use no curvature theorem.
\end{proof}

The proof of Lemma~\ref{lem:extensionvariance} applies the monotone-extension covariance argument in
Zhang and Xu \cite[proof of Theorem 1.3, especially (8), (15)--(18)]{ZhangXu}
with the forest extension decrement $3$ in place of their maximum-degree
decrement $\Delta+1$. Their result proves the corresponding extremal
variance bound for graphs of bounded maximum degree; Davies, Sandhu,
and Tan \cite[Section 1.2]{DST} introduced that extremal variance problem.
The present forest estimate is uniform over unbounded maximum degree.
Lemma~\ref{lem:scorevariance} gives a separate covariance bound for general graphs.
No claim of an optimal variance inequality or of priority for the
general covariance method is required here.

\section{Root-conditioning estimates}\label{sec:roots}
Root a tree and let $T_v$ be the descendant subtree at $v$, of order $m_v$.
Let $r_v$ be its root occupation odds, $q_v=r_v/(1+r_v)$ its root occupation
probability, and
\[
 y_v=\log(\lambda/r_v)=\sum_{w\text{ child of }v}\log(1+r_w).
\]
In particular, $r_v=\lambda e^{-y_v}$, $0<r_v\le\lambda$, and $y_v\ge0$.
Let $\delta_v$ be the occupied-root conditional mean minus the absent-root
conditional mean of the total independent-set size in $T_v$. Define
$\gamma_v$ similarly for the two conditional variances. The usual occupation-odds recursion is also fundamental to Weitz's
tree method \cite{Weitz}. Its first two logarithmic derivatives give
\cite[(4.1)--(4.2)]{FLNYZ}
\begin{align}
 \delta_v&=1-\sum_w q_w\delta_w,\label{eq:delta-rec}\\
 \gamma_v&=-\sum_wq_w\gamma_w-
             \sum_wq_w(1-q_w)\delta_w^2.\label{eq:gamma-rec}
\end{align}
They follow as well by applying $\lambda\,\mathrm d/\mathrm d\lambda$ once
and twice to $\log r_v$. At a leaf, $\delta_v=1$ and $\gamma_v=0$.

\subsection{The low-fugacity profile}
\begin{lemma}\label{low:lem:roots}
For $0<\lambda\le\lambda_*$ and a rooted tree of order $m$,
\[
 q(1-q)\delta^2\le\frac73\lambda m^{1/3},\qquad
 q(1-q)\gamma^2\le8\lambda m^{2/3}.
\]
\end{lemma}
\begin{proof}
Put $p=6/5$. Since $r=\lambda e^{-y}$ and $\log(1+r)\ge q$,
\[
 \frac{yq^p}{\log(1+r)}\le yr^{p-1}\le e^{-1}=\rho<1,
 \qquad q^p\le\log(1+r).
\]
If $Q_v(w)$ is the product of the descendant $q$'s along the path from
$v$ to $w$, with $Q_v(v)=1$, put $C_0=(1-\rho)^{-1}<2$. Induction gives
\[
 \sum_{w\in T_v}Q_v(w)^p\le1+C_0y_v\le1+2y_v.
\]
Indeed, the left side is $1+\sum_iq_i^p\sum_{w\in T_i}Q_i(w)^p$.
By induction and the preceding inequalities, this is at most
$1+(1+C_0\rho)\sum_i\log(1+r_i)=1+C_0y_v$.
Unroll the mean recurrence and apply H\"older to obtain
\[
 |\delta_v|\le m_v^{1/6}(1+2y_v)^{5/6}.
\]
The maximum of $e^{-y}(1+2y)^{5/3}$ occurs at $y=7/6$ and is less than
$7/3$. Cubing this last claim reduces it to
$e^{7/2}>100000/3087$; the exponential series through degree ten proves
that rational inequality. Multiplication by
$q(1-q)\le\lambda e^{-y}$ proves the first estimate.

The variance estimate follows from the same pointwise bound on
$\delta$. Put $u=3/2$ and $L_0=3/32$. Using
$q<2^{-10}$ gives
\[
 \frac{q^u(1+y)}{\log(1+r)}
 \le q^{u-1}+\rho q^{u-p}\le L_0,
\]
and
\[
 \frac{q^u(1+2y)^{2u/p}}{\log(1+r)}
 \le\lambda_*^{1/2}\sup_{y\ge0}e^{-y/2}(1+2y)^{5/2}
 =2^{5/2}e^{-9/4}<1.
\]
H\"older with exponents $3/2,3$ therefore proves inductively
$|\gamma_v|\le2m_v^{1/3}(1+y_v)^{2/3}$, because
$2L_0^{2/3}+1<2$. Finally
$4e^{-y}(1+y)^{4/3}<8$, which gives the second estimate.
\end{proof}

\begin{remark}[Complementary activity regimes]
The squared low-activity root exponents $1/3$ and $2/3$ are both below
one, so their centroid sums decrease with the terminal cutoff. The even
comparison uses this decay to obtain the substantially smaller threshold
$K$. For the whole interval $0<\lambda\le1$, the direct martingale
argument in Section~\MainRef{sec:centralcurvature} instead controls cubic
responses and the total predictable variance clock, without a bound on
the maximum degree or a superlinear root-variance profile.
\end{remark}

\section{Centroid conditioning at low activity}\label{sec:centroid}
Fix an integer $b\ge1$. In any remaining component of size greater than $b$,
reveal a centroid's occupation. If it is absent, delete it; if occupied,
record one occupied vertex and delete its closed neighborhood. Repeat until
every remaining component has size at most $b$. Fix all tie-breaking rules
in advance. Conditional on a reveal history, the remaining components are
independent hard-core forests at the same fugacity. This is the centroid
scheme of \cite[Section 5]{FLNYZ}.

If $m_t$ denotes the component size at reveal $t$, every child component has
at most half the parent's size. For any $0<\beta<1$, charging
$m_t^{\beta-1}$ to its vertices gives, on every history,
\begin{equation}\label{eq:halving}
 \sum_t m_t^\beta\le\frac{n b^{\beta-1}}{1-2^{\beta-1}}.
\end{equation}
For a fixed vertex, its component sizes form a halving chain until deletion
or termination; summing the geometric series backward from its last
processed component proves \eqref{eq:halving}. This reasoning applies to
each initial component and hence to every forest.

Let $M_t=\E(X\mid\mathcal F_t)$ and $S_t=\Var(X\mid\mathcal F_t)$,
where $\mathcal F_t$ is the information before reveal $t$. Let $M,S$ be the
terminal values, and write $\mathcal F_{\mathrm{term}}$ for the terminal
information. In particular, $M_0=\eta$, $S_0=\sigma^2$, and $\E M=\eta$.
With the root quantities of the affected component,
\begin{align}
 M_{t+1}-M_t&=\delta_t(\xi_t-q_t),\label{eq:martM}\\
 S_{t+1}-S_t&=\gamma_t(\xi_t-q_t)-q_t(1-q_t)\delta_t^2.\label{eq:martS}
\end{align}
Here $\xi_t$ is the centroid occupation indicator. These are the conditional
mean and total-variance identities in \cite[(5.4)]{FLNYZ}. Pad the sequence
by zero increments to length $n$ if it terminates early.

\subsection{Terminal bounds at low activity}
In this subsection $0<\lambda\le\lambda_*$. Set
\[
 A=\frac{19}{3},\qquad L=\frac{499}{500},\qquad U=\frac54.
\]

\begin{lemma}\label{low:lem:inputs}
The following estimates hold:
\begin{align}
 \E(M-\eta)^2&\le A\nu b^{-2/3},\label{low:eq:M}\\
 S-\sigma^2&=Z-D,\qquad \E Z=0,\qquad
 \E Z^2\le40\nu b^{-1/3},\label{low:eq:ZD}\\
 0&\le D\le A\nu b^{-2/3}\quad\text{on every history},\label{low:eq:D}\\
 L\nu&\le\sigma^2\le U\nu,\qquad 0\le S\le U\nu,\qquad
 \E S\le\sigma^2,\label{low:eq:scale}\\
 |\chi(t)|&\le e^{-t^2/25}\quad(|t|\le\pi\sigma),
 \qquad\chi(t)=\E e^{\ii t(X-\eta)/\sigma}.\label{low:eq:damping}
\end{align}
In particular,
\[
 \E|S-\sigma^2|\le\sqrt{40\nu}\,b^{-1/6}+A\nu b^{-2/3}.
\]
\end{lemma}
\begin{proof}
The martingale increments in \eqref{eq:martM} are orthogonal.
Lemma~\ref{low:lem:roots} and \eqref{eq:halving} therefore give
\[
 \E(M-\eta)^2\le\frac{7/3}{1-2^{-2/3}}\nu b^{-2/3}<A\nu b^{-2/3}.
\]
The same calculation bounds the total nonnegative drift $D$ in
\eqref{eq:martS} on every history. The centered sum $Z$ in that identity
satisfies
\[
 \E Z=0,\qquad \E Z^2\le
 \frac8{1-2^{-1/3}}\nu b^{-1/3}<40\nu b^{-1/3}.
\]
Here $(7/3)/(1-2^{-2/3})<19/3$ follows from
$2^{-2/3}<12/19$, verified by cubing. This proves the first three
estimates, and Cauchy--Schwarz gives the bound on $\E|S-\sigma^2|$.

Lemma~\ref{lem:extensionvariance} gives
$\sigma^2\ge\nu/(1+3\lambda)^2>L\nu$.
For the upper bound use a separate centroid cutoff $256$.
Each terminal component of order $m$ has second moment at most
$m\lambda+m(m-1)\lambda^2$: for distinct vertices,
$\Prob(u,v\in I)\le\lambda^2$, by removing the two specified vertices
from a weighted independent set. Since $256^{2/3}>40$,
\[
 \sigma^2=\E S+\Var M
 \le\left(1+\frac{255}{3125}+\frac{19}{120}\right)\nu<U\nu.
\]
Apply this universal variance bound to the residual forest on each
fixed history to obtain $S\le U\nu$. Total variance gives
$\E S\le\sigma^2$. Finally, for $|t|\le\pi\sigma$, use
$\sin(|t|/(2\sigma))\ge |t|/(\pi\sigma)$ in
Lemma~\MainRef{lem:bipartition}, together with $\sigma^2\le U\nu$ and
$2U(22/7)^2(3126/3125)^3<25$, to prove the damping estimate.
\end{proof}

\subsection{A uniform normal-approximation consequence}
\begin{corollary}\label{thm:clt}
There is an absolute constant $C$ such that every nonempty forest,
$0<\lambda\le\lambda_*$, and $\nu\ge1$ satisfy
\[
 \sup_{z\in\mathbb R}\left|
 \Prob_\lambda\left(\frac{X-\eta}{\sigma}\le z\right)-\Phi_0(z)\right|
 \le C\nu^{-1/8},
\]
where $\Phi_0$ is the standard normal distribution function.
\end{corollary}

\begin{proof}
Use the low-activity constants in Lemma~\ref{low:lem:inputs}, and
choose $b=\lceil\nu^{3/8}\rceil$. Conditional on the terminal history, $X$
is a known integer plus a sum of independent random variables $X_i$ with
$0\le X_i\le b$. Their total conditional variance is $S$, and
\[
 \sum_i\E\!\left(\left|X_i-\E(X_i\mid\mathcal F_{\mathrm{term}})\right|^3\,\middle|\,\mathcal F_{\mathrm{term}}\right)\le bS.
\]
On $S\ge\sigma^2/2\ge\nu/8$, the Berry--Esseen inequality
\cite[Theorem 1.1, dimension one]{Raic} bounds the conditional normal
approximation error by $C b/\sqrt\nu$. For variance ratio $s\ge1/2$ and
shift $h$, the normal distribution functions satisfy
\[
 \sup_z|\Phi_0((z-h)/\sqrt s)-\Phi_0(z)|\le C(|h|+|s-1|).
\]
This follows by differentiating the normal distribution function in shift
and scale. Apply it with $h=(M-\eta)/\sigma$ and $s=S/\sigma^2$.
The complement $S<\sigma^2/2$ has probability at most
$2\E|S-\sigma^2|/\sigma^2$. Lemma~\ref{low:lem:inputs} therefore gives,
writing $d_K$ for Kolmogorov distance,
\[
 d_K\left(\frac{X-\eta}{\sigma},\mathcal N(0,1)\right)
 \le C\left(\frac b{\sqrt\nu}+b^{-1/3}
             +\nu^{-1/2}b^{-1/6}+b^{-2/3}\right).
\]
The choice of $b$ makes this at most $C\nu^{-1/8}$.
No step requires $\lambda$ to be bounded away from zero.
\end{proof}

\section{An even characteristic-function comparison}\label{low:sec:even}
\begin{lemma}\label{low:lem:even}
Let independent random variables $Y_i$ each have range at most $b$, put
$v_i=\Var Y_i$, $S=\sum_i v_i$, and
$\Phi(s)=\E\exp(\ii s\sum_i(Y_i-\E Y_i))$.
For $|s|\le1/b$,
\begin{align}
 |\operatorname{Im}\Phi(s)|
 &\le\frac{16}{85}bS|s|^3 e^{-11s^2S/24},\label{low:eq:imag}\\
 |\operatorname{Re}\Phi(s)-e^{-s^2S/2}|
 &\le b^2\left(\frac4{85}S s^4+\frac2{111}S^2s^6\right)
            e^{-11s^2S/24}.\label{low:eq:real}
\end{align}
\end{lemma}
\begin{proof}
Write $\phi_i=a_i+\ii d_i$ and $g_i=e^{-s^2v_i/2}$. Symmetrization and
$1-\cos z\ge11z^2/24$ for $|z|\le1$ give
$|\phi_i|\le e^{-11s^2v_i/24}$; also $g_i$ obeys this weaker damping.
The centered range bound implies $v_i\le b^2/4$ and
\[
 |d_i|\le b v_i|s|^3/6.
\]
Both remainders $a_i-(1-s^2v_i/2)$ and
$g_i-(1-s^2v_i/2)$ are nonnegative. Their respective upper bounds are
$b^2v_i s^4/24$ and $v_i^2s^4/8\le b^2v_i s^4/32$.
Consequently
\[
 |a_i-g_i|\le b^2v_i s^4/24.
\]

Telescope the product $\prod_i\phi_i$ against the real product
$\prod_i a_i$. Each imaginary term is bounded by $|d_i|$ times damped
remaining factors. The missing variance costs at most $e^{11/96}$.
Since $e^{11/96}\le96/85$, this proves \eqref{low:eq:imag}.

Next telescope against $\prod_i g_i$, retaining the real part. The real
individual errors contribute at most
\[
 \frac1{24}e^{11/96}b^2S s^4e^{-11s^2S/24}.
\]
Each imaginary individual error is paired with the imaginary part of a
preceding product. Applying the previous product telescope to that
preceding product bounds their total by
\[
 \frac1{36}e^{11/48}b^2s^6
       \sum_{j<i}v_iv_j\,e^{-11s^2S/24}
 \le\frac1{72}e^{11/48}b^2S^2s^6e^{-11s^2S/24}.
\]
This includes all correlations among the real and imaginary factors;
no independence beyond that of the $Y_i$ is being added.
Using $e^{11/48}\le48/37$ gives \eqref{low:eq:real}.
\end{proof}

\section{Local curvature at low fugacity}\label{sec:lowcurvature}
Throughout this section $0<\lambda\le\lambda_*$, and
$A=19/3$, $L=499/500$, $U=5/4$ are the constants in
Lemma~\ref{low:lem:inputs}. The good-event threshold is chosen explicitly;
this retains the conditional mean and variance errors while reducing
the uniform starting rank.

\begin{theorem}\label{low:thm:parametric}
Suppose $\eta=k$ is an integer. For any integer $b\ge376$, put
$T=\sigma/b$ and $p_j=\Prob(X=j)$. Then
\begin{align*}
 &\left|\sigma^3(2p_k-p_{k-1}-p_{k+1})-\frac1{\sqrt{2\pi}}\right|\\
 &\quad\le\frac{12863}{1000}b^{-2/3}
 +\frac{4417}{1000}\frac{b^2}{\nu}
 +\frac{9013}{1000}\frac{b^{2/3}}{\sqrt\nu}\\
 &\qquad+\frac{6423}{1000}\nu^{-1/2}b^{-1/6}
 +2563\nu^{1/2}b^{-10/3}+\nu^{-1}
 +\frac{43}{10}\left(T+\frac{25}{2T}\right)e^{-T^2/25}.
\end{align*}
\end{theorem}
\begin{proof}
Set
\[
 \gamma=\frac{81}{100},\qquad \delta=\frac{3409}{50000},
 \qquad G=\{S\ge\gamma\sigma^2\}.
\]
The exact rational inequality
\[
 376^2>\left(\frac{A}{L(1-\gamma-\delta)}\right)^3
\]
implies
$D\le A\nu b^{-2/3}<(1-\gamma-\delta)\sigma^2$.
Since $S-\sigma^2=Z-D$, the event $G^c$ forces
$Z<-\delta\sigma^2$. Chebyshev's inequality and
\eqref{low:eq:ZD} therefore give
\begin{equation}\label{low:eq:bad}
 \Prob(G^c)\le\frac{40}{\delta^2L^2}\nu^{-1}b^{-1/3}.
\end{equation}

On a fixed history set $W=M-\eta$, and let $\Phi$ be the conditional
characteristic function of $X-M$. The residual components are
conditionally independent and have ranges at most $b$, so
Lemma~\ref{low:lem:even} applies. For $s=t/\sigma$,
\begin{align*}
 |\operatorname{Re}(e^{\ii sW}\Phi(s))-e^{-s^2S/2}|
 &\le|\operatorname{Re}\Phi(s)-e^{-s^2S/2}|\\
 &\quad+\frac{s^2W^2}{2}e^{-s^2S/2}
       +|sW|\,|\operatorname{Im}\Phi(s)|.
\end{align*}
No independence between $W$ and $\Phi$ is assumed. Instead,
\eqref{low:eq:scale} gives
\begin{equation}\label{low:eq:cross}
 \E S^2\le U\nu\sigma^2,\qquad
 \E(|W|S)\le\sqrt{\E W^2\,\E S^2}
 \le\sqrt{AU}\,\nu\sigma b^{-1/3}.
\end{equation}
Put $a_0=11\gamma/24=297/800$. On $G$ the block damping is
at most $e^{-a_0t^2}$, and the conditional Gaussian damping is at
most $e^{-\gamma t^2/2}$. Moreover,
\[
 |e^{-t^2S/(2\sigma^2)}-e^{-t^2/2}|
 \le\frac{t^2}{2\sigma^2}|S-\sigma^2|e^{-\gamma t^2/2}.
\]
Averaging on $G$, using \eqref{low:eq:cross}, and charging at most
$2\Prob(G^c)$ on the remaining histories proves, for $|t|\le T$,
\begin{align*}
 |\operatorname{Re}\chi(t)-e^{-t^2/2}|
 &\le\left(\frac4{85L}\frac{b^2}{\nu}t^4
       +\frac{2U}{111L^2}\frac{b^2}{\nu}t^6
       +\frac{16\sqrt{AU}}{85L^{3/2}}
              \frac{b^{2/3}}{\sqrt\nu}t^4\right)e^{-a_0t^2}\\
 &\quad+\left(\frac{A}{L}b^{-2/3}
       +\frac{\sqrt{40}}{2L}\nu^{-1/2}b^{-1/6}\right)
           t^2e^{-\gamma t^2/2}+2\Prob(G^c).
\end{align*}
The coefficient $A/L$ includes both the conditional mean-phase error
and the variance drift.

At integer mean, Fourier inversion uses the real part:
\[
 \sigma^3(2p_k-p_{k-1}-p_{k+1})
 =\frac1{2\pi}\int_{-\pi\sigma}^{\pi\sigma}
      h_\sigma(t)\operatorname{Re}\chi(t)\,dt,
 \quad h_\sigma(t)=2\sigma^2(1-\cos(t/\sigma)).
\]
Here $0\le h_\sigma\le t^2$ and
$|h_\sigma-t^2|\le t^4/(12\sigma^2)$.
For even $r$, set
\[
 J_r(z)=\frac1{2\pi}\int_{\mathbb R}t^r e^{-zt^2}\,dt.
\]
The four principal integrated constants satisfy
\begin{align*}
 C_M&=\frac ALJ_4(\gamma/2)<\frac{12863}{1000},\\
 C_N&=\frac{\sqrt{40}}{2L}J_4(\gamma/2)<\frac{6423}{1000},\\
 C_B&=\frac4{85L}J_6(a_0)+\frac{2U}{111L^2}J_8(a_0)
       <\frac{4417}{1000},\\
 C_C&=\frac{16\sqrt{AU}}{85L^{3/2}}J_6(a_0)
       <\frac{9013}{1000}.
\end{align*}
These inequalities follow by rational arithmetic from
$\pi>333/106$ and
\[
 J_4(z)=\frac{3}{8z^2\sqrt{\pi z}},\quad
 J_6(z)=\frac{15}{16z^3\sqrt{\pi z}},\quad
 J_8(z)=\frac{105}{32z^4\sqrt{\pi z}}.
\]
The certificate replaces each square root by an upper rational
enclosure and verifies the comparison by squaring. The bounds
$333/106<\pi<22/7$ can themselves be checked from
$\pi=16\arctan(1/5)-4\arctan(1/239)$ using the alternating series
through ten and eleven terms.

The exceptional-history integral is at most
\[
 \frac{2\Prob(G^c)T^3}{3\pi}
 \le\frac{80U^{3/2}}{3\pi\delta^2L^2}
          \nu^{1/2}b^{-10/3}
 <2563\nu^{1/2}b^{-10/3}.
\]
For this numerical comparison one may use
$U^{3/2}<1747/1250$ and $\pi>333/106$; the coarser bound
$U^{3/2}<7/5$ is not used. The Gaussian multiplier error is at most
$1/(4\sqrt{2\pi}\sigma^2)<\nu^{-1}$.

Finally, Lemma~\ref{low:lem:inputs} gives
$|\chi(t)|\le e^{-t^2/25}$ on its full integration interval.
Integration by parts, followed by the elementary Gaussian tail bound,
yields for the two outside tails
\begin{align*}
 \frac1\pi\int_T^\infty t^2(e^{-t^2/25}+e^{-t^2/2})\,dt
 &\le\frac1\pi\left[
 \frac{25}{2}\left(T+\frac{25}{2T}\right)e^{-T^2/25}
 +(T+1/T)e^{-T^2/2}\right]\\
 &<\frac{43}{10}\left(T+\frac{25}{2T}\right)e^{-T^2/25}.
\end{align*}
The last inequality uses $27/(2\pi)<43/10$.
This proves the stated parametric estimate.
\end{proof}

\subsection{The explicit starting rank}
\begin{proof}[Proof of Theorem~\ref{low:thm:main}]
Put $u=\nu^{1/8}$ and choose the constants
\[
 K=31\,500\,000,\qquad
 u_0=\frac{8655433}{1000000},\qquad c=\frac{293}{500},
 \qquad d=c+u_0^{-3},\qquad T_0=\frac{147168}{10000}.
\]
Let $b=\lceil cu^3\rceil$. Exact rational comparisons give
\[
 u_0^8<K,\qquad cu_0^3>375,\qquad T_0^2d^2<Lu_0^2.
\]
Thus $u\ge u_0$, $b\ge376$, $cu^3\le b\le du^3$, and
$T\ge\sqrt L\,u/d\ge T_0$. The error in
Theorem~\ref{low:thm:parametric} is bounded above by
\begin{align*}
 E_0={}&\frac{12863}{1000}\frac{c^{-2/3}}{u_0^2}
 +\frac{4417}{1000}\frac{d^2}{u_0^2}
 +\frac{9013}{1000}\frac{d^{2/3}}{u_0^2}
 +\frac{6423}{1000}c^{-1/6}u_0^{-9/2}\\
 &+2563\frac{c^{-10/3}}{u_0^6}+u_0^{-8}
 +\frac{43}{10}\left(T_0+\frac{25}{2T_0}\right)e^{-T_0^2/25}.
\end{align*}
Every power replaced by its $u_0$ value decreases with $u$. The tail
function also decreases for every $T>0$, since
\[
 \frac{d}{dT}\left[(T+25/(2T))e^{-T^2/25}\right]
 =-\left(\frac{25}{2T^2}+\frac{2T^2}{25}\right)e^{-T^2/25}<0.
\]
Hence this is a uniform estimate for all $\nu\ge K$.

Upper rational enclosures for the five radicals in $E_0$, verified
by integer powers, and the lower bound
\[
 e^{T_0^2/25}>\sum_{r=0}^{60}\frac{(T_0^2/25)^r}{r!}
\]
give the exact comparison $E_0<398127/1000000$.
Moreover,
\[
 \frac{19943}{50000}<\sqrt{\frac7{44}}<\frac1{\sqrt{2\pi}},
 \qquad
 \frac{19943}{50000}-\frac{398127}{1000000}>\frac1{2000}.
\]
This proves the claimed positive curvature margin.

The arithmetic--geometric mean inequality converts positive second
difference into strict log-concavity, and untilting preserves the
determinant sign. Lemma~\MainRef{lem:mean} gives
\[
 \frac{n\lambda}{1+3\lambda}\le\eta_F(\lambda)\le n\lambda.
\]
Since the mean is continuous and strictly increasing,
every $j\le n/3128$ has a saddle $\lambda\le1/3125$.
For $j\ge K$, its value $\nu=n\lambda$ satisfies $\nu\ge j\ge K$.
This proves the full rank assertion.
\end{proof}

\section{Sharper thresholds and explicit barrier bounds}
\label{sec:quantitativecompletion}
\subsection{Completing the unconditional through-mean conclusion}
\begin{proof}[Proof of Corollary~\ref{central:cor:throughmean}]
Let $n\ge N_0=3125\cdot2^{104}$ and let $j\le\eta_F(1)$ be an internal
integer rank. If $j<K$, Lemma~\MainRef{lem:extension} applies because
\[
 2j^2-4<2K^2<2^{51}<N_0.
\]
Otherwise choose the unique physical activity $0<\lambda\le1$ with
$\eta_F(\lambda)=j$. Existence and uniqueness follow from continuity,
$\eta_F(\lambda)\to0$ as $\lambda\downarrow0$, and
$\mathrm d\eta_F/\mathrm d\log\lambda=\Var_\lambda X>0$.
If $\lambda\le1/3125$, then $n\lambda\ge j\ge K$ by
Lemma~\MainRef{lem:mean}, so Theorem~\ref{low:thm:main} applies.
If $\lambda>1/3125$, then $n\lambda>n/3125\ge2^{104}$ and
Theorem~\MainRef{central:thm:central} applies. Every saddle is covered.
Finally, $3125<2^{12}$ gives $N_0<2^{116}$, and
Lemma~\MainRef{lem:mean} gives $\eta_F(1)\ge n/4$.
\end{proof}

\begin{proof}[Proof of Corollary~\ref{cor:smalltilt}]
The general extension bound $(j+1)i_{j+1}\le(n-j)i_j$ implies
$n>(B+1)j+B$ at a stated tilted ascent. If $j<K$, then $B\ge2j-3$,
so $n>2j^2-3$ and Lemma~\MainRef{lem:extension} applies.
If $j\ge K$, then $n>3128j$, and Theorem~\ref{low:thm:main} applies.
No all-order barrier or unimodality theorem is used.
\end{proof}

\subsection{A quantitatively sharper prefix assembly}
\begin{proof}[Alternative proof of Theorem~\MainRef{thm:main}]
Fix $\varepsilon_0=1/3128$. By \cite[Theorem 1.2 and Section 9]{FLNYZ},
there is an absolute $N_{\varepsilon_0}$ such that every forest of order
$n\ge N_{\varepsilon_0}$ is strictly log-concave at every integer
$\varepsilon_0n\le j\le17\alpha(F)/25$. Take
\[
 N\ge\max\{N_{\varepsilon_0},\,2K^2,\,3128K\},
 \qquad K=31\,500\,000.
\]
For $1\le j<K$, apply Lemma~\MainRef{lem:extension}.
For $K\le j\le n/3128$, apply Theorem~\ref{low:thm:main}.
For $n/3128\le j\le17\alpha(F)/25$, apply the cited theorem.
Every integer belongs to at least one of these ranges. The added
requirement $n\ge3128K$ makes the low-density interval nonempty.
The constant $\varepsilon_0$ is fixed; no $n$-dependent value of
$\varepsilon$ is substituted into the external theorem.
\end{proof}

This gives an explicit lower-rank contribution to $N$, but it does not
replace $N_{\varepsilon_0}$ by $N_0$. The through-mean theorem and
the prescribed-prefix theorem have different endpoints.

\section{Optional explicit barrier bound: a fixed external input}\label{sec:lvz}
We state precisely the additional source dependence. The above-mean
argument uses the matching-mode, occupation, and large-mean weighted
estimates in the stated proof of Li, Vallier, and Zhang \cite{LVZ}.
For clarity, the anchor weights are marginal occupation probabilities:
put $\pi_v=\Prob_\lambda(v\in I)$ and choose an independent set $A$
maximizing $\sum_{v\in A}\pi_v$. Bipartiteness gives
$\sum_{v\in A}\pi_v\ge\eta/2$. Conditional on $Y=I\cap(V(F)- A)$,
the free vertices $A- N(Y)$ are independent Bernoulli choices
with parameter $q=\lambda/(1+\lambda)$. Thus their number $W$ satisfies
\[
 q\E W=\E|I\cap A|\ge\eta/2.
\]
In particular, $\E W\ge j/(2q)$ at a saddle with mean $j$.
The following input records precisely the consequences used here. It is
assumed from the fixed source and its stated computational estimates; it
is not a new theorem proved by the analytic argument of this paper.

\begin{externalinput}[LVZ, fixed draft 2.2]\label{input:lvz}
Let the physical activity be chosen so that $\eta_F(\lambda)=j$, and let
$A$ maximize the sum of the actual occupation marginals as above.
Under the estimates of \cite{LVZ}, the following statements hold:
\begin{enumerate}
\item An ascent $i_{j+1}>i_j$ has saddle activity $\lambda<7/3$,
by the matching comparison, matching-mode bound, and occupation bound
in \cite[Proposition 2.3(b), Lemma 3.7, and Theorem 3.9]{LVZ}.
\item In the domain $1\le\lambda<7/3$ and $\E W\ge50$, the stated
forest moment bounds and large-mean weighted inequality in
\cite[Sections 4--5, Theorem 5.5 and the averaging step of Proposition 5.8]{LVZ} give
\begin{equation}\label{eq:lvz-weighted}
 2i_j>(3-4q)i_{j-1}+(4q-1)i_{j+1}.
\end{equation}
\end{enumerate}
\end{externalinput}

\paragraph{Extraction from the fixed source.}
The second statement uses the proof of the source theorem, not just its
no-valley conclusion. Throughout its Section~5.3 the fixed kernel
weights are $w_L=q^2(3-4q)$ and $w_R=(1-q)^2(4q-1)$. The strict
averaging margin gives
\[
 \left(w_L\frac{1-q}{q}+w_R\frac{q}{1-q}\right)p_j
       -w_Lp_{j-1}-w_Rp_{j+1}>0.
\]
Substitute $p_r=i_r\lambda^r/Z_F(\lambda)$ and
$\lambda=q/(1-q)$, then divide by the positive factor
$q(1-q)\lambda^j/Z_F(\lambda)$ to obtain
\eqref{eq:lvz-weighted}. The source's starting free-count means all
lie in $[19,50]$, so $\E W\ge50$ suffices for this part of its proof.
These are external inputs, not conclusions of our Fourier analysis.
The source's finite moment checks supporting its weighted inequality
are treated as external computational inputs. We use only these stated inputs,
not the source's announced universal unimodality conclusion.

\paragraph{Version and data identification.}
All external result numbers in this section refer to draft 2.2 of
\path{arXiv:2610.07943v1}, dated October 6, 2026.
The fixed supplementary-material snapshot identified in its Section~7 is
\begin{center}\small\ttfamily
9f4f1ecfc439fdf43f8cfad0e888e54fe778d428.
\end{center}
The repository is
\href{https://github.com/zhangzenozhang-jpg/forest-unimodality-v2.2-supplement/tree/9f4f1ecfc439fdf43f8cfad0e888e54fe778d428}{\texttt{zhangzenozhang-jpg/forest-unimodality-v2.2-supplement}}.
Its three named assets are \texttt{twin\_v2.2\_data.txt},
\texttt{supplementary\_data.zip}, and \texttt{supplementary\_code.zip}.
The source-reported checksums and the exact statement locations are recorded
in the accompanying \texttt{anc/EXTERNAL\_INPUT\_LVZ.json}. These source-reported identifiers fix the assumed external materials.
The accompanying programs do not verify those external assets, interval
calculations, or formal development.
In particular, this section uses the weighted-kernel averaging estimate,
not merely the announced all-forest unimodality theorem.

\begin{proposition}[Above-mean barrier consequence of the stated input]\label{prop:above-mean}
Assuming External input~\ref{input:lvz}, a forest of order $n\ge280$
satisfies $\Delta_j>0$ at every internal ascent index
$j\ge\eta_F(1)$.
\end{proposition}
\begin{proof}
Choose the physical saddle with mean $j$. Since $j\ge\eta_F(1)$,
its activity is at least one. The ascent input gives
$1\le\lambda<7/3$, hence $1/2\le q<7/10$.
Lemma~\MainRef{lem:mean} gives $j\ge n/4\ge70$, and the free-count
bound gives $\E W\ge j/(2q)>50$. Thus
\eqref{eq:lvz-weighted} applies. Put $A_q=3-4q$ and $B_q=4q-1$;
then $0<A_q\le B_q$ and $A_q+B_q=2$.
If $y=i_{j+1}/i_j>1$ and $\Delta_j\le0$, then
$i_{j-1}/i_j\ge1/y$, so the right side of
\eqref{eq:lvz-weighted}, divided by $i_j$, is at least
\[
 A_q/y+B_qy=2+(y-1)(B_q-A_q/y)>2,
\]
a contradiction.
\end{proof}

\begin{corollary}[Explicit eventual barrier bound]\label{cor:explicit-barrier}
With the external input in Proposition~\ref{prop:above-mean}, every
forest of order at least $N_0=3125\cdot2^{104}$ is strictly barrier log-concave.
\end{corollary}
\begin{proof}
Corollary~\ref{central:cor:throughmean} gives strict log-concavity for
$j\le\eta_F(1)$, and Proposition~\ref{prop:above-mean} handles ascents
at the remaining ranks. Since $N_0>280$, both apply.
\end{proof}

The displayed cutoff is conservative and remains far beyond a feasible
finite enumeration. It does not settle the assertion at all orders and
does not give the same cutoff for the prescribed $17\alpha/25$ prefix.
The low-fugacity, through-mean, and uniform small-tilt conclusions do not
use the external estimates of this section.

\begin{corollary}[Explicit forest cutoff with the fixed external input]\label{cor:wc-explicit}
Assuming External input~\ref{input:lvz}, every well-covered forest $F$
with $\mu(F)\ge N_0=3125\cdot2^{104}$ has a strictly log-concave
independence polynomial. Equivalently, its nonisolated part has order
at least
\[
 2N_0=6250\cdot2^{104}=3125\cdot2^{105}.
\]
In particular, order at least $2N_0$ suffices for well-covered forests
without isolated vertices, and hence for well-covered trees.
\end{corollary}
\begin{proof}
Use Corollary~\ref{cor:explicit-barrier} for the base forest $H$ in
the proof of Theorem~\MainRef{cor:eventual-wc} and apply Theorem~\MainRef{thm:corona-transfer} with
transform order $|V(H)|+z(F)$. This is the explicit, additionally
conditional route; it does not replace the existential full-prefix
threshold by $N_0$.
\end{proof}

\section{Further consequences}
\paragraph{Many tail failures are compatible with the prefix.}
Bautista-Ramos \cite[Corollary 3]{BautistaSpherical} proves that, for every
fixed $m\ge3$ and all sufficiently large odd $\ell$, the subcubic tree
$T_{m,\ell}=T(2^m1^\ell)$ has at least $J_{m-3}$ log-concavity failures,
where $J_r=(2^r-(-1)^r)/3$. His threshold for $\ell$ may depend on $m$.
The breaks certified by that proof have reflected offsets
$1\le u\le J_m-2$. Counting vertices and using
\cite[Lemma 15]{BautistaSpherical}, write
$N_T=2^m(\ell+2)-1$ and $d=\alpha(T_{m,\ell})=(N_T+J_m)/2$.
For $\ell=2k+1$, one has
$d=5J_m+2^m k+(3(-1)^m-1)/2\ge5J_m-2$.
Thus their original indices satisfy
$j=d-u\ge d-J_m+2>4d/5>17d/25$.
This locates the certified breaks only, not every possible failure.
Combining his unbounded-break result with Theorem~\MainRef{cor:barrier}
shows that strictly barrier-log-concave subcubic trees can have arbitrarily many
log-concavity failures: first choose $m$, then choose odd $\ell$ large enough
for both statements. This consequence uses his stated bound, not the sharper
exponent-count observation recorded separately. Spherical graph symmetry is
not the coefficient symmetry used in the following subsection.

\subsection{Symmetry and \texorpdfstring{$\gamma$}{gamma}-positivity}\label{subsec:symmetry}
The symmetry considered here is coefficient symmetry:
$a_j=a_{d-j}$ for a sequence $a_0,\ldots,a_d$. In this subclass, the
barrier condition of Theorem~\MainRef{cor:barrier} is not weaker than
log-concavity.

\begin{lemma}[Reflection and the barrier condition]\label{lem:symmetric-barrier}
Let $a_0,\ldots,a_d$ be a positive symmetric sequence. Its log-concavity
determinants satisfy
\[
 \Delta_j:=a_j^2-a_{j-1}a_{j+1}=\Delta_{d-j}
 \qquad(1\le j<d).
\]
The sequence is barrier log-concave if and only if it is log-concave.
Log-concavity at all indices $1\le j\le\lfloor d/2\rfloor$ implies
log-concavity at every internal index; the same statement holds with strict
inequalities.
\end{lemma}
\begin{proof}
The determinant identity follows by substituting $a_j=a_{d-j}$.
Suppose the barrier condition holds and $\Delta_j<0$. Its contrapositive
gives $a_{j+1}\le a_j$. Positivity and
$a_{j-1}a_{j+1}>a_j^2$ then give $a_{j-1}>a_j$.
At the reflected index $d-j$, one has
$a_{d-j+1}=a_{j-1}>a_j=a_{d-j}$ and $\Delta_{d-j}<0$, contradicting
the barrier condition. The reverse implication is immediate.
For the last assertion, reflect each internal index $j$ to
$\min\{j,d-j\}\le\lfloor d/2\rfloor$.
When $d\le1$, all assertions about internal indices are vacuous.
\end{proof}

\begin{corollary}[An explicit symmetric case]\label{cor:symmetric-forest}
Every forest with at least $N_0=3125\cdot2^{104}$ vertices and a symmetric
independence polynomial is strictly log-concave at every internal index.
This conclusion is independent of the input in Section~\ref{sec:lvz}.
\end{corollary}
\begin{proof}
For a positive symmetric sequence of degree $d$, pairing ranks $j$ and
$d-j$ gives $\eta_F(1)=d/2$. Corollary~\ref{central:cor:throughmean} gives
strict log-concavity through $\lfloor d/2\rfloor$. Reflect the remaining
indices using Lemma~\ref{lem:symmetric-barrier}.
\end{proof}

No $\gamma$-positivity hypothesis is needed. Combined with the existence
theorem of Hibi, Kara, and Vien \cite[Theorem 0.1]{HKV}, this gives a
tree with a symmetric strictly log-concave independence polynomial at
every sufficiently large order and at every sufficiently large
polynomial degree. The existence result is theirs; the strict
log-concavity conclusion follows from the preceding corollary. For
example, any degree $d\ge N_0$ supplied by their theorem comes from
a tree with at least $d$ vertices. No all-order classification is asserted.

\paragraph{The distinction from $\gamma$-positivity.}
For a symmetric polynomial $h(z)=\sum_{j=0}^d a_jz^j$ with $a_0>0$,
write its unique expansion as
\[
 h(z)=\sum_{r=0}^{\lfloor d/2\rfloor}\gamma_r z^r(1+z)^{d-2r}
      =(1+z)^d\Gamma_h\!\left(\frac{z}{(1+z)^2}\right).
\]
The condition $\gamma_r\ge0$ implies symmetry and unimodality
\cite[Definition 1.3 and Lemma 1.4]{HKV}, but does not imply the barrier
condition. For example,
\[
 (1+z)^4+11z^2=1+4z+17z^2+4z^3+z^4
\]
is $\gamma$-positive, yet $4^2-1\cdot17=-1$ at a strict ascent.
This is a positive-polynomial example, not a tree independence polynomial.

Conversely, the $13$-vertex tree $G_{13}$ in
\cite[Lemma 2.3]{HKV} has coefficient sequence
\[
 (1,13,66,176,279,279,176,66,13,1).
\]
Its $\gamma$-polynomial is
$\Gamma_{G_{13}}(u)=1+4u+2u^2-2u^3-u^4$, whereas its internal
determinants are
\[
 (103,2068,12562,28737,28737,12562,2068,103).
\]
Thus it is strictly log-concave, and hence barrier log-concave, but is not
$\gamma$-positive. These direct calculations use the coefficient list
printed in \cite{HKV}; they are not claims made by its Bridge Lemma.
Thus neither $\gamma$-positivity nor barrier log-concavity implies the other
for positive symmetric polynomial sequences. The first
nonimplication above is not asserted within the narrower class of tree
independence polynomials. In contrast, pre-Lorentzianity is a sufficient
certificate for ordinary log-concavity on the classes where it is proved
\cite{BendjeddouHardiman,LiuTang}.

\begin{remark}[A sharper approximation at larger order]
The stronger approximation error $1/1000$ at $n\ge2^{256}$ and
$1/3125\le\lambda\le1$ also follows from the new parametric bound
(\MainRef{cubic:eq:12}). Indeed, $\nu>2^{244}$; choose $T=128$.
The first three terms are respectively less than
$2^{-38}$, $2^{-72}$, and $2^{-242}$. Since
$128^2/250>64\log2$ and
$(250/3)(128+125/128)<2^{15}$, the tail is less than $2^{-49}$.
Their sum is less than $1/1000$, and $T<\pi\sigma$ follows from
$\sigma^2\ge\nu/32$. The improved order cutoff $N_0$ uses the positive
curvature margin in Theorem~\MainRef{central:thm:central}, rather than this
smaller approximation error.
\end{remark}

\section{Domains of validity and quantitative margins}\label{app:checkability}
The following tables summarize the domains and numerical margins used
in the proof. All probabilistic expectations retain their stated sampling
law; none of the comparisons permits an exchange of rank and incidence
weights.

\subsection{Scope at the proof-critical transitions}
\begin{center}\small
\begin{tabular}{@{}>{\raggedright\arraybackslash}p{.21\textwidth}>{\raggedright\arraybackslash}p{.35\textwidth}>{\raggedright\arraybackslash}p{.36\textwidth}@{}}\toprule
Object or step&Required scope&What must not be substituted\\\midrule
Extension averages&Uniform independent sets of a supported rank in a forest;
$m_\alpha=0$&Unequal-activity rank weights, or the extension-incidence predecessor law.\\[3pt]
Response norms&Descendant probabilities, a parent-before-child orientation,
and $0<\lambda_v\le1$&Whole-graph marginals, unrestricted activities,
or a graph with cycles.\\[3pt]
Clock comparison&Centered martingale increments, finite third absolute
moments, positive target variance; crossing determined predictably&The
random clock replaced by its expectation, or an independent suffix.\\[3pt]
Centroid argument&The entire reveal history; fixed tie-breaking; common
activity on each remaining component&Independence between the conditional
mean and characteristic function.\\[3pt]
Fourier curvature&Integer mean, span-one integer-valued size, nonempty forest&An
$o(\sigma^{-1})$ local limit bound without second-difference control.\\[3pt]
Prefix assembly&Fixed $\varepsilon$ in the FLNYZ input and the two proved
lower-rank regimes&Setting $\varepsilon=1/n$ in a fixed-density theorem.\\[3pt]
Explicit BL&External input~\ref{input:lvz}, with $1\le\lambda<7/3$,
$\E W\ge50$&The source's headline unimodality claim alone.\\[3pt]
Well-covered consequence&$|V(H)|=\mu(F)$ and
$|V(F)|-z(F)=2\mu(F)$; require $\mu(F)\ge N$&Replacing the matching-number
condition by large total order when isolated vertices are allowed.\\\bottomrule
\end{tabular}
\end{center}
The notation is local: $q_v$ denotes a descendant occupation probability;
$\pi_v$ in Section~\ref{sec:lvz} is an actual marginal. The predictable
clock $A$, an independent set also denoted $A$, and the numerical constant
$A=19/3$ belong to different sections. The symbol $\nu=n\lambda$ denotes
activity mass; $\mu(F)$ denotes maximum matching size in the
well-covered-forest application.
A reader should attach the section and probability law before comparing
similarly named quantities.

\subsection{Numerical reserve, not a robustness theorem}
The central proof reconstructs the rational majorant
\[
 E_* = \frac{613959936395331007118784021397819}
 {1597239756287569045885846038773760}
 =0.38438808824940923\ldots .
\]
It leaves $0.39-E_*=0.00561191175059\ldots$ to the printed error bound
and $(2\pi)^{-1/2}-E_*=0.01455419215202\ldots$ to positivity. An extra
\emph{additive error in that same normalized estimate} must fit the
corresponding reserve. A relative change in a lemma constant must first
be propagated through the proof; it cannot be compared directly with
this absolute error.

For the sharper low-activity proof, the reserve beyond the claimed
$1/2000$ margin is
\[
 (2\pi)^{-1/2}-\frac{398127}{10^6}-\frac1{2000}
       =0.00031528040143268\ldots .
\]
All radicals and exponential remainders used for the endpoint bound are
controlled in the written proof. Neither decimal above is a replacement
for its exact comparisons. For the tail step, the endpoint-fraction
separation is $17/25-2/3=1/75$; the exact integer step remains
$j\le\lceil(2\alpha-1)/3\rceil-1$.

\subsection{Interpretation of the numerical bounds}
An additive perturbation must be measured in the normalization of the
corresponding estimate. A change to a response-norm bound propagates
through both the cubic-increment and variance-clock estimates before it
enters the Fourier integral. The exact comparisons in the proof, rather
than rounded decimal margins, determine the stated cutoffs. The finite
programs supplied with the paper check these constants and selected
identities, whereas the uniformity in order, degree, and activity is
established analytically.

\paragraph{Reproduction.}
The accompanying code checks 51 exact rational guards and finite
weighted-tree, martingale, and corona identities. From the extracted
source directory, run
\begin{verbatim}
python3 anc/run_checks.py --output-dir ../p1-checks
\end{verbatim}
These finite checks do not replace the analytic proofs. The external
LVZ computational and formal developments are not run by this command.
The exact certificates of the sharp-prefix companion \cite{LevitDraft}
are not premises of the main theorems here.

\section*{Use of AI tools}
ChatGPT was used for mathematical exploration and proof development,
code preparation and checking, literature comparison, and language
editing. The author is responsible for the mathematical content,
attribution, and computational materials. This statement does not
assert formal verification.

%% file: figures/prefix_cover.tex
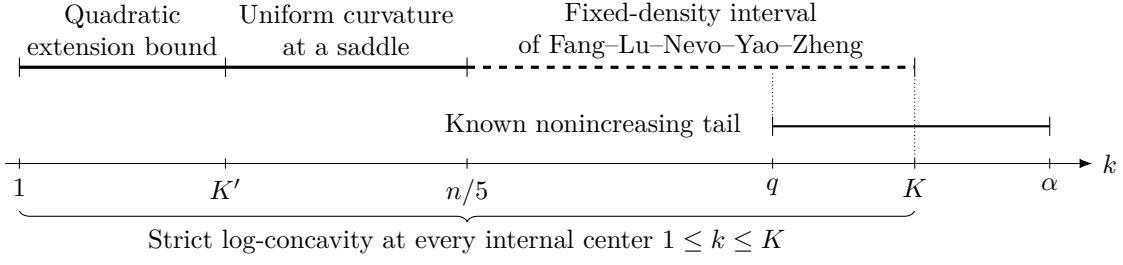
\begin{figure}[tbp]
\centering
\begin{tikzpicture}[x=.94cm,y=.85cm,>=Latex,font=\small]
\draw[->] (0,0)--(15.3,0) node[right] {$k$};
\foreach \x/\t in {.2/1,3.1/{K'},6.5/{n/5},10.8/q,12.8/K,14.7/\alpha}
 {\draw (\x,.08)--(\x,-.08);\node[below] at (\x,-.08) {$\t$};}
\draw[very thick] (.2,1.5)--(3.1,1.5);
\draw[very thick] (3.1,1.5)--(6.5,1.5);
\draw[very thick,dashed] (6.5,1.5)--(12.8,1.5);
\foreach \x in {.2,3.1,6.5,12.8}{\draw (\x,1.35)--(\x,1.65);}
\node[align=center] at (1.65,2.08) {Quadratic\\extension bound};
\node[align=center] at (4.8,2.08) {Uniform curvature\\at a saddle};
\node[align=center] at (9.65,2.08) {Fixed-density interval\\of Fang--Lu--Nevo--Yao--Zheng};
\draw[thick] (10.8,.58)--(14.7,.58);
\draw (10.8,.43)--(10.8,.73);\draw (14.7,.43)--(14.7,.73);
\node[anchor=east] at (10.5,.6) {Known nonincreasing tail};
\draw[densely dotted] (10.8,.75)--(10.8,1.42);
\draw[densely dotted] (12.8,.08)--(12.8,1.35);
\draw[decorate,decoration={brace,mirror,amplitude=4pt}] (.2,-.74)--(12.8,-.74);
\node[below,align=center] at (6.5,-.87) {Strict log-concavity at every internal center $1\le k\le K$};
\end{tikzpicture}
\caption{The three ranges in the proof of Theorem~\ref{thm:main},
shown schematically and not to scale. Here $K'=2^{104}$,
$K=\lfloor17\alpha/25\rfloor$, and $q=\lceil(2\alpha-1)/3\rceil$.
The order threshold makes the first two ranges overlap; the protected
prefix reaches past the beginning of the nonincreasing tail. Every ascent
therefore has a strictly positive Tur\'an determinant.}
\label{fig:prefix-cover}
\end{figure}

%% file: figures/corona_decomposition.tex
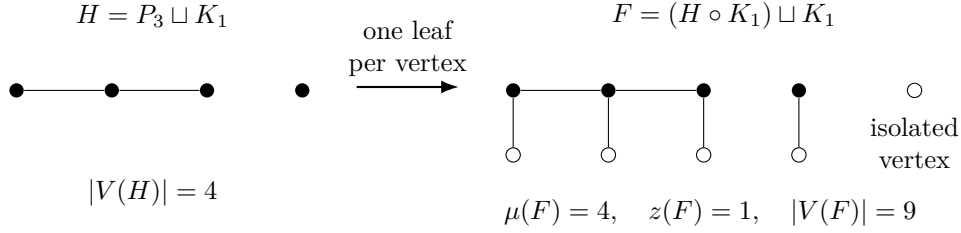
\begin{figure}[tbp]
\centering
\begin{tikzpicture}[x=.9cm,y=.9cm,>=Latex,font=\small,
 old/.style={circle,fill=black,inner sep=1.9pt},
 leaf/.style={circle,draw,fill=white,inner sep=1.9pt}]
\node at (2,1.1) {$H=P_3\sqcup K_1$};
\foreach \i/\x in {a/0,b/1.4,c/2.8,d/4.2}{\node[old] (l\i) at (\x,0) {};}
\draw (la)--(lb)--(lc);
\draw[->,thick] (5,.05)--(6.5,.05) node[midway,above,align=center] {one leaf\\per vertex};
\node at (10.4,1.1) {$F=(H\circ K_1)\sqcup K_1$};
\foreach \i/\x in {a/7.3,b/8.7,c/10.1,d/11.5}{
 \node[old] (r\i) at (\x,0) {};
 \node[leaf] (p\i) at (\x,-.95) {};
 \draw (r\i)--(p\i);
}
\draw (ra)--(rb)--(rc);
\node[leaf] (iso) at (13.2,0) {};
\node[below,align=center] at (13.2,-.25) {isolated\\vertex};
\node[align=center] at (2,-1.5) {$|V(H)|=4$};
\node[align=center] at (10.2,-1.8) {$\mu(F)=4,\quad z(F)=1,\quad |V(F)|=9$};
\end{tikzpicture}
\caption{The well-covered-forest decomposition. Filled vertices are the
base vertices; open vertices are attached leaves or isolated components.
In general, $F=(H\circ K_1)\sqcup sK_1$ has
$I(F;x)=(1+x)^{|V(H)|+s}I(H;x/(1+x))$ and
$|V(F)|-z(F)=2\mu(F)=2|V(H)|$. The drawing illustrates the representation,
not the large-order threshold.}
\label{fig:corona-decomposition}
\end{figure}